\documentclass{article}
\pdfoutput=1

\usepackage{graphicx} %
\usepackage{amsmath}
\usepackage{amsthm}
\usepackage{amssymb}   
\usepackage{mathrsfs}  
\usepackage{enumitem}  
\usepackage{fullpage}
\usepackage[colorlinks,
            linkcolor=red,
            anchorcolor=blue,
            citecolor=green
            ]{hyperref}
\usepackage{chngcntr}  
\numberwithin{equation}{section}
            
\usepackage{mathtools}

\usepackage{algorithm}
\usepackage{algorithmic}
\usepackage[numbers]{natbib}

\usepackage{microtype}
\usepackage{graphicx}
\usepackage{subfigure}
\usepackage{booktabs} %
\usepackage{makecell}

\usepackage{hyperref}

\newcommand{\bx}{\mathbf{x}}
\newcommand{\by}{\mathbf{y}}
\newcommand{\bz}{\mathbf{z}}
\newcommand{\ba}{\mathbf{a}}
\newcommand{\bb}{\mathbf{b}}

\newcommand{\bu}{\mathbf{u}}

\newcommand{\bg}{\mathbf{g}}

\newcommand{\gtil}{\widetilde{G}}
\newcommand{\btil}{\widetilde{B}}
\newcommand{\ghat}{\widehat{G}}
\newcommand{\res}{\mathrm{res}}
\newcommand{\out}{\text{out}}
\newcommand{\bsf}{\mathsf{B}}

\newcommand{\E}{\mathbb{E}}

\renewcommand{\phi}{\varphi}

\newcommand{\good}{\mathsf{Good}}
\newcommand{\bad}{\mathsf{Bad}}

\DeclareFontFamily{OT1}{pzc}{}
\DeclareFontShape{OT1}{pzc}{m}{it}{<-> s * [1.200] pzcmi7t}{}
\DeclareMathAlphabet{\mathpzc}{OT1}{pzc}{m}{it}

\DeclareMathOperator{\dist}{dist}

\DeclareMathOperator{\prox}{prox}

\DeclareMathOperator{\gra}{gra}
\newcommand{\id}{\mathrm{Id}}

\usepackage{xcolor,colortbl}

\hypersetup{
  linkcolor  = blue,
  citecolor  = teal,
  colorlinks = true,
}
\makeatletter
\newcommand\blfootnote[1]{%
  \begingroup
  \renewcommand{\@makefntext}[1]{\noindent\makebox[1.8em][r]#1}
  \renewcommand\thefootnote{}\footnote{#1}%
  \addtocounter{footnote}{-1}%
  \endgroup
}
\makeatother

\usepackage[capitalize,noabbrev]{cleveref}

\theoremstyle{plain}
\newtheorem{theorem}{Theorem}[section]

\newtheorem{lemma}[theorem]{Lemma}

\theoremstyle{definition}

\newtheorem{assumption}[theorem]{Assumption}
\theoremstyle{remark}
\newtheorem{remark}[theorem]{Remark}
\theoremstyle{fact}
\newtheorem{fact}[theorem]{Fact}

\usepackage[textsize=tiny]{todonotes}
\usepackage[bottom=1in, top=1in, left=1in, right=1in]{geometry}

\title{Solving Stochastic Variational Inequalities \\ without the Bounded Variance Assumption}
\author{Ahmet Alacaoglu\footnote{Department of Mathematics, University of British Columbia. \url{alacaoglu@math.ubc.ca}} 
\and Jun-Hyun Kim\footnote{University of British Columbia. \url{junhyun@student.ubc.ca}}}
\date{}

\begin{document}

\maketitle

\begin{abstract}
  We analyze algorithms for solving stochastic variational inequalities (VI) without the bounded variance or bounded domain assumptions, where our main focus is min-max optimization with possibly unbounded constraint sets. We focus on two classes of problems: monotone VIs; and structured nonmonotone VIs that admit a solution to the \emph{weak Minty VI}. The latter assumption allows us to solve structured nonconvex-nonconcave min-max problems. For both classes of VIs, to make the expected residual norm less than $\varepsilon$, we show an oracle complexity of $\widetilde{O}(\varepsilon^{-4})$, which is the best-known for constrained VIs. In our setting, this complexity had been obtained with the bounded variance assumption in the literature, which is not even satisfied for bilinear min-max problems with an unbounded domain. We obtain this complexity for stochastic oracles whose variance can grow as fast as the squared norm of the optimization variable.
\end{abstract}

\section{Introduction}
In this work, we focus on stochastic variational inequalities (SVI) where the aim is to
\begin{equation}\label{eq: svi_prob}
\text{find~} \bz^\star \text{~s.t.~} \langle G(\bz^\star), \bz - \bz^\star  \rangle + r(\bz) - r(\bz^\star) \geq 0~~~\forall \bz,
\end{equation}
where $r\colon\mathbb{R}^m \to \mathbb{R}\cup \{ +\infty \}$ is a proper, convex and closed function and $G\colon\mathbb{R}^m\to \mathbb{R}^m$ is an operator. When $r$ is equal to the indicator function of a convex and closed set $C\subset \mathbb{R}^m$, \eqref{eq: svi_prob} reduces to a more well-studied SVI problem with set constraints.
In the stochastic case, we assume that we have an \emph{unbiased} oracle $\gtil$ such that
\begin{equation}\label{eq: unbiased}
\mathbb{E}[\gtil(\bz)] =  G(\bz).
\end{equation}
One common application of SVI is stochastic min-max optimization, formulated as,
\begin{equation}\label{eq: minmax_prob}
\min_{\bx\in \mathbb{R}^d} \max_{\by\in \mathbb{R}^n} f(\bx, \by) + h_1(\bx) - h_2(\by),
\end{equation}
where $h_1, h_2$ are given \emph{regularizers}. Taking $h_1, h_2$ as indicator functions results in constrained min-max optimization.
This problem maps to \eqref{eq: svi_prob} with $\bz=\binom{\bx}{\by}$,
\begin{equation*}
G(\bz)=\binom{\nabla_\bx f(\bx, \by)}{-\nabla_\by f(\bx, \by)} \text{~and~} r(\bz) = h_1(\bx) + h_2(\by).
\end{equation*}
The problem \eqref{eq: minmax_prob} gained significant interest in machine learning recently due to applications in adversarial and robust learning, as well as generative adversarial networks \citep {duchi2021learning, madry2018towards, goodfellow2014generative}. Moreover, \eqref{eq: minmax_prob} is a classical framework for solving constrained optimization problems \citep{bertsekas2014constrained}. For example, given a nonlinear programming problem
\begin{equation}\label{eq: nonlp}
\min_{\bx \in X} p(\bx) \text{~subject to~} q(\bx) \leq 0,
\end{equation} 
the standard approach is to use the Lagrangian duality framework to reformulate this problem as
\begin{equation}\label{eq: lagran_nonlp}
\min_{\bx \in X}\max_{\by \geq 0} p(\bx) + \langle \by, q(\bx) \rangle.
\end{equation} 
It is now easy to see that this problem corresponds to \eqref{eq: minmax_prob}, which itself is a special case of \eqref{eq: svi_prob}.
In addition to min-max optimization, SVI is also a common framework to model  problems arising in game theory, see for example \citep{mertikopoulos:tel-02428077}.\\[2mm]
\textbf{Variance assumptions.} In our stochastic setup, a common assumption in addition to \eqref{eq: unbiased} is that of bounded variance \citep{nemirovski2009robust,robbins1951stochastic}. In particular, $\gtil$ satisfies this assumption if the following holds for a finite $\sigma$:
\begin{equation}\label{eq: bdd_vr}\tag{BV}
\E \| \gtil(\bz) - G(\bz)\|^2 \leq \sigma.
\end{equation}
Despite being standard, it is well-known that this is a restrictive assumption. In particular, let us consider a simple bilinear min-max problem, in view of \eqref{eq: minmax_prob}. In this case, the operator $G$ from \eqref{eq: svi_prob} is linear, and hence \eqref{eq: bdd_vr} fails unless the domains of $h_1, h_2$ are bounded. Boundedness of the domains of $h_1, h_2$ is unfortunately unrealistic, which, for example, can be seen by considering \eqref{eq: nonlp} with linear $p, q$ (which indeed gives us the classical linear programming problem), where the resulting  min-max problem \eqref{eq: lagran_nonlp} has unbounded domains.
We also refer to \cite{iusem2017extragradient} for a discussion about this assumption in our setting.

In this work, we focus on algorithms and convergence analyses for solving SVI under the weaker assumption
\begin{equation}\label{eq: bg_def}
\E \| \gtil(\bz) - G(\bz)\|^2 \leq B^2 \| \bz-\bz^\star \|^2 + \sigma^2,
\end{equation}
and its \emph{equivalent} variant\footnote{The equivalence is trivial to see by using Young's inequality on the right-hand side and adjusting the definitions of constants $B^2, \sigma^2$. See Appendix \ref{app: prelim} (Fact~\ref{lem: bg_equiv}) for details.}
\begin{equation}\label{eq: bg_def2}
\E \| \gtil(\bz) - G(\bz)\|^2 \leq B^2 \| \bz-\bz_0 \|^2 + \sigma^2,
\end{equation}
which are used for example in \cite{gladyshev1965stochastic,iusem2017extragradient,choudhury2023single,alacaoglu2025towards,kotsalis2022simpleii}.
It is simple to see that this assumption is trivially satisfied when $G$ is linear, which corresponds to $f$ in \eqref{eq: minmax_prob} being bilinear. To our knowledge, this is currently the most relaxed variance assumption under which one gets optimal convergence rate and complexity guarantees even for convex minimization, see for example \citep{alacaoglu2025towards,neu2024dealing}.
In particular, this assumption implies the other relaxations of \eqref{eq: bdd_vr}, such as the ones appearing in \cite{khaled2023unified,khaled2022better}.\\[2mm]
\textbf{Notation. }
The subdifferential of the convex function $r$ in \eqref{eq: svi_prob} is denoted as $\partial r$. With this, for notational convenience, we write \eqref{eq: svi_prob} as an \emph{inclusion} problem:
\begin{equation}\label{eq: inc_prob}
\text{find~} \bz^\star \text{~such that~} 0 \in (G+\partial r)\bz^\star.
\end{equation}
Classical references containing the background for this formulation include \cite{bauschke2017convex, rockafellar1976monotone}. Recall the definition of the proximal operator: $\prox_r(\bx)=\arg\min_\bu r(\bu) + \frac{1}{2} \| \bu-\bx\|^2$. We show results on $\bz^{out}$ selected uniformly at random from the iterates.

\textbf{Residual.} Despite the ubiquity of SVI and the well-known limitations of the bounded variance assumption \eqref{eq: bdd_vr}, analysis of algorithms for SVI under \eqref{eq: bg_def}, especially for nonmonotone problems, remained mostly unexplored. This is relevant especially in the case when we do not have bounded domains (since a bounded domain would reduce \eqref{eq: bg_def} to \eqref{eq: bdd_vr}). Unbounded domains are common in min-max optimization, as mentioned earlier. In this setting, we will focus on complexity guarantees on the \emph{residual}, given as
\begin{equation}\label{eq: res_def}
\res(\bz_k):=\dist(0, (G+\partial r)\bz_k) = \min_{\bu \in (G+\partial r)\bz_k} \| \bu\|,
\end{equation}
which is the generalization of the gradient (or operator) norm optimality measure (which is for unconstrained problems). Indeed, one can observe that when $r\equiv 0$, this reduces to $\|G(\bz_k)\|$. For the min-max case, this is nothing but $\sqrt{\| \nabla_\bx f(\bx_k, \by_k) \|^2 + \| \nabla_\by f(\bx_k, \by_k) \|^2}$.

The reason for our focus on the residual is the following: When we have nonconvex-nonconcave problems, the duality gap cannot be used. Even with convex-concave problems, in the case of unbounded domains, the duality gap, which is commonly used for VIs (see \cite{nemirovski2009robust,facchinei2003finite}), is not applicable, since it would require taking a maximum over the domain of $r$ which is unbounded. 

A common workaround is the \emph{so-called} restricted duality gap \citep{nesterov2007dual}. The restricted gap is also not completely satisfactory because for the restricted gap to be an optimality measure, one needs the knowledge of the norm of the iterates that the algorithm generates \citep{nesterov2007dual}. This is neither known in advance, nor easily controllable in the stochastic case since stochastic algorithms generate sequences that are not uniformly bounded. \\[2mm]
\textbf{Role of convexity and monotonicity.}
Let us consider the case of convex-concave min-max optimization, which is the problem we have when $f(\cdot, \by)$ is convex and $f(\bx, \cdot)$ is concave. This leads to a monotone SVI, that is, the operator $G$ defined by the gradients of $f$ is a monotone operator:
\begin{equation*}
\langle G(\bx) - G(\by), \bx-\by \rangle \geq 0.
\end{equation*}
To go beyond the convex-concave case, we will consider a common assumption from the literature, which requires the existence of a solution to the ($\rho\geq 0$)-weak Minty Variational Inequality (MVI), that is,
\begin{align}
\langle \bu, \bx - \bx^\star \rangle \geq -\rho \| \bu\|^2, \text{~where~} 
(\bx, \bu) \in \gra (G+\partial r) = \{ (\bx, \bu): \bu \in (G+\partial r)\bx \}.\label{eq: weakminty}\tag{wMVI}
\end{align} 
When $r\equiv 0$, this assumption was proposed by \cite{diakonikolas2021efficient} which generalized the comonotonicity assumption of \cite{combettes2004proximal, bauschke2021generalized} (see also \cite{pethick2022escaping} for the constrained version) which is shown to hold for the so-called \emph{interaction dominant} min-max problems \citep{grimmer2023landscape} and others \citep{pethick2022escaping}. For brevity, we sometimes refer to this as the \emph{weak MVI assumption} in the sequel. This defines a class of \emph{nonmonotone} problems since monotonicity of $G$ is not required. The level of nonmonotonicity (or, nonconvex-nonconcavity for the special case of min-max problems in \eqref{eq: minmax_prob}) is set by the parameter $\rho\geq0$.

When $\rho=0$, this reduces to the well-known assumption of the existence of a solution to the Minty VI \citep{facchinei2003finite}, which is also sometimes referred to as the \emph{coherence}, see, e.g., \cite{mertikopoulos2019optimistic}. For policy optimization in reinforcement learning, a variation of this with $\rho=0$ holds \citep{lan2023policy}. Some applications, including those with $\rho > 0$ are included in \citep{pethick2022escaping, alacaoglu2024revisiting,lee2021fast}.

The assumption  \eqref{eq: weakminty} is among the weakest-known requirement under which complexity results have been shown for nonconvex-nonconcave problems. Some alternative assumptions require Polyak-\L{}ojasiewicz (P\L{}) or Kurdyka-\L{}ojasiewicz (K\L{})-type properties to hold for the dual variable, see, e.g., \cite{li2025nonsmooth}. The relationship between the latter class of problems and problems satisfying \eqref{eq: weakminty} is not well-understood and the algorithm development between these two classes of problems have been largely independent of each other. We focus on analyzing algorithms under \eqref{eq: weakminty}.

\subsection{Assumptions}
For all our results except \Cref{sec: vr}, we require a standard Lipschitzness assumption for the operator $G$, which is defined as
\begin{equation*}
\| G(\bx) - G(\by) \| \leq L \|\bx-\by\|.
\end{equation*}
We now collect our assumptions used in the sequel.
\begin{assumption}\label{asp: 1}
Let $G$ be $L$-Lipschitz and let $G+\partial r$ satisfy \eqref{eq: weakminty} with a solution $\bz^\star$.
\end{assumption}
In \Cref{sec: vr}, we will use the stronger \emph{expected Lipschitzness} assumption that we discuss more in the sequel.
\begin{assumption}\label{asp: 2}
We can get unbiased samples $\gtil(\bx, \xi)$ and $\gtil(\by, \xi)$ with the same random seed $\xi$. Let $G$ be $L$-expected Lipschitz, that is:
\begin{equation*}
\E_{\xi} \| \gtil(\bx, \xi) - \gtil(\by, \xi) \|^2 \leq L_{\mathrm{exp}}\| \bx-\by\|^2,
\end{equation*}
where $\E [\gtil(\bx, \xi)] = G(\bx)$.
\end{assumption}
We finally formalize our assumptions concerning $G$.
\begin{assumption}\label{asp: 3}
At each iteration, we receive an i.i.d. oracle $\gtil$ such that
\begin{align*}
\E[\gtil(\bz)] = G(\bz) \text{~~~and~~~} 
\E \| \gtil(\bz) - G(\bz)\|^2 \leq B^2 \| \bz-\bz_0\|^2 + \sigma^2.
\end{align*}
\end{assumption}
Clearly, this assumption can also be restated to hold only at $\bz=\bz_t$. As mentioned before, this remains the weakest variance assumption under which optimal complexity guarantees are shown for convex minimization or monotone VIs \citep{alacaoglu2025towards}. In view of the last two assumptions, let us point out that we overload the notation $\gtil$ and use it without the argument $\xi$ when the context is suitable, that is, we often use the notation $\gtil(\bx_k) := \gtil(\bx_k, \xi_k)$.

\subsection{Contributions and Comparisons}
In the sequel, we first state the algorithms which are either existing in the literature or are simple modifications over the existing algorithms. Our main contribution is \emph{the analysis of these methods without the bounded variance or bounded domain assumptions}. We next highlight our main complexity results (by which we mean the number of stochastic first-order oracles ---\emph{sfo}--- used by an algorithm) for obtaining an output $\bz$ for which we have $\E[\res(\bz)]\leq\varepsilon$ (see \eqref{eq: res_def}).
\begin{itemize}
\item We show, in \Cref{sec: mb_main}, that for $\rho < \frac{1}{12L}$, a forward-backward-forward algorithm with mini-batching achieves the best-known stochastic oracle complexity $\widetilde{O}(\varepsilon^{-4})$.

\item We show, in \Cref{sec: mlmc}, that for $\rho < 1/L$ (which is the tightest-known upper bound for the nonmonotonicity parameter $\rho$), an inexact fixed-point algorithm equipped with multilevel Monte Carlo (MLMC) estimator achieves the best-known (expected) complexity $\widetilde{O}(\varepsilon^{-4})$.

\item We show, in \Cref{sec: vr}, that for $\rho < \frac{1}{16L}$ (see Thm. \ref{th: vr} for the details), a variance reduced forward-backward-forward method with Halpern anchoring achieves the best-known complexity $\widetilde{O}(\varepsilon^{-4})$. This method is single-loop and only uses $3$ stochastic oracles for $G$ at every iteration, not requiring any large mini-batch sizes.

\item In \Cref{sec: numer}, we test the numerical performance of our algorithms and illustrate two main points: \emph{(i)} as predicted by theory, \Cref{alg:weakmvi_stoc} converges for a wider range of $\rho$ compared to earlier works, for a problem where even the \emph{deterministic} extragradient algorithm, without noise, diverges \citep{gorbunov2023convergence}, \emph{(ii)} introduction of Halpern anchoring in \Cref{alg: var_red_fbf} results in a more robust behavior with respect to tuning of the initial step size compared to the benchmark  method in \citep{pethick2023solving}.

\end{itemize}

To our knowledge, we provide the first complexity results where nonmonotonicity in view of \eqref{eq: weakminty} can be tolerated for constrained min-max problems without bounded variance. \Cref{tab:constraints} contains a comparison between our results and the existing results for SVI without bounded variance.\emph{Our results match the complexity results known with bounded variance for Halpern-based or variance reduced methods, see \cite{lee2021fast,pethick2023solving}.}
\begin{table*}[t!]
\centering
\begin{tabular}{l l l l l l l}
\toprule
&\makecell{\textbf{Constraint}} & \makecell{\textbf{Need to} \\ \textbf{know$^\dagger$}} & \makecell{\textbf{Range}\\\textbf{of $\rho$}} & \makecell{\textbf{Complexity}} & \textbf{Lipschitz} & \makecell{\textbf{Single loop \&} \\ \textbf{MB$^*$-free}} \\
\midrule
\cite{choudhury2023single} & $\times$ & \makecell{$L, B, K, \sigma_\star^2,$ \\ $\| \bz_0-\bz^\star\|^2$ } & $ \rho < \frac{1}{2L}$ & $O(\varepsilon^{-4})$ & Asp. \ref{asp: 1} & $\times$ (large MB) \\[3mm]
\cite{kotsalis2022simpleii} & $\checkmark$ & $L, B, K$ & $\rho = 0$ & $O(\varepsilon^{-4})$ & Asp. \ref{asp: 1} & $\times$ (large MB) \\[3mm]
\cite{iusem2017extragradient} & $\checkmark$ & $L, B$ & $\rho = 0$ & $\widetilde{O}(\varepsilon^{-4})$ & Asp. \ref{asp: 1} & $\times$ (large MB) \\[3mm]
\cite{alacaoglu2025towards} & $\checkmark$ & $L, B$ & $\rho = 0$& $\widetilde{O}(\varepsilon^{-4})$ & Asp. \ref{asp: 2} & $\checkmark$\\[3mm]
Thm. \ref{th: mb}&$\checkmark$ & $L, B$ & \makecell{$\rho < \frac{1}{12L}$}& $\widetilde{O}(\varepsilon^{-4})$ & Asp. \ref{asp: 1} & $\times$ (incr. MB)\\[3mm]
Thm. \ref{th: mlmc} &$\checkmark$ & $L, B$ & $\rho < \frac{1}{L}$ & $\widetilde{O}(\varepsilon^{-4})$ & Asp. \ref{asp: 1} & $\times$ (loops) \\[3mm]
Thm. \ref{th: vr}&$\checkmark$ & $L, B$ & \makecell{$\rho <\frac{1}{16L}^\ddagger$} & $\widetilde{O}(\varepsilon^{-4})$ & Asp. \ref{asp: 2} & $\checkmark$\\
\bottomrule
\end{tabular}
\caption{Existing results without bounded variance assumption. $^*$MB: mini-batch. $^\dagger$Constants that algorithms need to set parameters. $^\ddagger$ The upper bound of $\rho$ in this case converges to $\frac{1}{16L}$ as a step size parameter get smaller, see \Cref{th: vr_supp} for details.}
\label{tab:constraints}
\end{table*}
\paragraph{Discussion about \Cref{tab:constraints}.} In the nonmonotone case ($\rho > 0$ in \Cref{tab:constraints}), our main contributions are two-fold, \emph{(i)} we provide complexity results for constrained SVI without bounded variance assumptions \emph{(ii)} even in the unconstrained case, we improve the existing result of \cite{choudhury2023single} because we show the results with the best-known range on $\rho$ and also without the knowledge of uncomputable quantities about the solution, that were required in \citep[Thm. 4.5]{choudhury2023single} for setting the mini-batch size.

In the monotone case, which is implied by the results with $\rho=0$ in \Cref{tab:constraints}, our contributions are the following: compared to \cite{kotsalis2022simpleii,iusem2017extragradient}, we provide guarantees without large mini-batch sizes.
 Compared to \cite{alacaoglu2025towards}, we provide guarantees under \Cref{asp: 1} which is weaker than \Cref{asp: 2} which requires a multi-point access to the oracle with a stronger Lipschitzness assumption. 

A simple example of a $G$ that satisfies \Cref{asp: 1} but not \Cref{asp: 2} (see~\cite{alacaoglu2024revisiting}) is $G(x) = G_1(x) + G_2(x)$ where $G_1(x) = x^2$ and $G_2(x) = -x^2$ and $\gtil$ is selected uniformly at random between $G_1, G_2$. However, \Cref{asp: 2} remains needed for single-loop algorithms even with bounded variance \cite{pethick2023solving}.

As we will discuss further in the sequel, three main results we prove are complementary: that is, they each extend the state-of-the-art in different directions and none of the proposed methods uniformly improve over the others. We refer to \Cref{tab:constraints} for a summary. The proofs are deferred to the appendices.

\section{Results with Minibatching}\label{sec: mb_main}
We start with the most straightforward approach one may take to address our problem: an algorithm with large mini-batches. On a high level, let us remark that when the mini-batch size becomes large enough, the algorithm behaves more and more like a \emph{full-gradient} algorithm. As a warm-up, we start with analyzing such an algorithm. In fact, this simple approach has complementary advantages to the other approaches we consider in the sequel. In particular, we require the weakest set of assumptions in this result and unlike our results in \Cref{sec: mlmc}, the number of oracles used at each iteration is \emph{deterministic} rather than \emph{expected}. We discuss this further in \Cref{sec: mlmc}.
\subsection{Algorithmic Ideas}
An idea that we use throughout the paper is the forward-backward-forward (FBF) algorithm of \cite{tseng2000modified}, which iterates as
\begin{align*}
\mathsf{z}_{k+1/2} &= \prox_{\eta_k r}(\mathsf{z}_k - \eta_k G(\mathsf{z}_k)) \\
\mathsf{z}_{k+1} &= \mathsf{z}_{k+1/2} - \eta_k \left( G(\mathsf{z}_{k+1/2}) - G(\mathsf{z}_k) \right).
\end{align*}
One can easily extend this method to the stochastic case:
\begin{equation}\label{eq: fbf_stoc}
\begin{aligned}
\bz_{k+1/2} &= \prox_{\eta_k r}(\bz_k - \eta_k \ghat(\bz_k)) \\
\bz_{k+1} &= \bz_{k+1/2} - \eta_k \left( \ghat(\bz_{k+1/2}) - \ghat(\bz_k) \right),
\end{aligned}
\end{equation}
where 
\begin{align}\label{eq: fbf_mb}
\ghat(\bz_k) &= \frac{1}{b_k} \sum_{i=1}^{b_k} \gtil(\bz_k, \xi_k^{i}) \text{~and~} \\	\ghat(\bz_{k+1/2}) &= \frac{1}{b_k} \sum_{i=1}^{b_k} \gtil(\bz_{k+1/2}, \xi_{k+1/2}^{i}),
\end{align}
for i.i.d. samples $\xi_k^i$ and $\xi_{k+1/2}^i$ for $i=1,\dots,b_k$.

This algorithmic construction is certainly not new, see, for example \citep{bohm2022two} who analyzed a similar stochastic FBF  under bounded variance, for complexity results on the expected gap function. Another related work by \cite{iusem2017extragradient} analyzed an extragradient method \'a la \cite{korpelevich1976extragradient} under \Cref{asp: 3} with mini-batching, in the case where $\rho=0$. 
\subsection{Complexity Analysis}\label{subsec: mb_complex}
We start with our main complexity result of this section. See \Cref{th: mb_supp} for a precise statement.
\begin{theorem}\label{th: mb}
Let Assumptions \ref{asp: 1} and \ref{asp: 3} hold and suppose that $\rho < \frac{1}{12L}$.
For the algorithm in \eqref{eq: fbf_stoc} with gradient estimators computed as \eqref{eq: fbf_mb} with $b_k = \Theta(k\log(k+1))$ and $\eta_k = \Theta(1/L)$ (where precise parameters are given in Thm. \ref{th: mb_supp}), we have that
\begin{align*}
\E [\res(\bz^{\out})] \leq \varepsilon \text{~with sfo complexity~} \widetilde{O}(\varepsilon^{-4}),
\end{align*}
where $\bz^{\out}$ is generated by selecting an index $\hat k$ uniformly at random after running $K$ iterations and letting $\bz^\out = \bz_{\hat{k}+1/2}$.
\end{theorem}
\begin{remark}
The main limitation of this scheme is that it requires increasing mini-batch sizes and the upper bound for $\rho$ is suboptimal (which we did not optimize). However, its strength lies in the fact that the number of stochastic oracles used at each iteration is deterministic. Compared to \citep[Thm. 4.5]{choudhury2023single}, we can handle constrained problems and our batch sizes do not require any knowledge about the solution\footnote{Parameter $b_k$, given explicitly in Thm. \ref{th: mb_supp} depends on $L, \rho, B$ whereas the mini-batch size in \citep[Thm. 4.5]{choudhury2023single} also depends on uncomputable $\|\bx_0-\bx^\star\|^2$ and $\mathbb{E}\|\nabla f_i(\bx^\star)\|^2$.}. Compared to \cite{kotsalis2022simpleii, iusem2017extragradient}, we can handle $\rho > 0$.
\end{remark}
The main idea of this proof is similar to \cite{kotsalis2022simpleii} and \cite{iusem2017extragradient} with the exception that our analysis can tolerate a nonzero $\rho$. In particular, by using the specific form of the mini-batch size, one can first show that the iterates stay bounded in expectation:
\begin{equation*}
\E \| \bz_k - \bz^\star \|^2 \leq R^2,
\end{equation*}
for an explicit $R$ provided in the Appendix \ref{app: proofs2}.

Then, one can use this bound on \eqref{eq: bg_def} to further upper bound the variance with a term depending on $R$. After this, the main analysis of FBF goes through. The main reason that we can tolerate a nonzero $\rho$ is the following chain of identities that follow from the definitions of $\bz_{k+1/2}$, $\bz_{k+1}$: 
\begin{align*}
&\bz_{k+1/2} = \arg\min_{\bz} r(\bz) + \frac{1}{2\eta_k} \| \bz-(\bz_k-\eta_k \ghat(\bz_k)) \|^2 \\
 &\iff \bz_{k+1/2} + \eta_k \partial r(\bz_{k+1/2}) \ni \bz_k - \eta_k \ghat(\bz_k) \\
 &\iff  G(\bz_{k+1/2}) + \partial r(\bz_{k+1/2}) \ni \eta_k^{-1}(\bz_k - \bz_{k+1}) + G(\bz_{k+1/2}) - \ghat(\bz_{k+1/2}).
\end{align*}
Then, in view of the assumption \eqref{eq: weakminty}, we have
\begin{align*}
\langle \eta_k^{-1}(\bz_k - \bz_{k+1}), \bz_{k+1/2} - \bz^\star \rangle &+\langle G(\bz_{k+1/2}) - \ghat(\bz_{k+1/2}), \bz_{k+1/2} - \bz^\star \rangle \\
&\quad \geq - \rho \| \eta_k^{-1}(\bz_k - \bz_{k+1}) + G(\bz_{k+1/2}) - \ghat(\bz_{k+1/2})\|^2,
\end{align*}
as $(\bz_{k+1/2}, \eta_k^{-1}(\bz_k - \bz_{k+1}) + G(\bz_{k+1/2}) - \ghat(\bz_{k+1/2})) \in \gra(G+\partial r)$.

Analyzing FBF in the standard way from this,
one has
\begin{align}
\E \| \bz_{k+1} - \bz^\star\|^2 &\leq \E \| \bz_{k} - \bz^\star\|^2+ \frac{2\rho}{\eta_k} \E\| \bz_k - \bz_{k+1} + \eta_k(G(\bz_{k+1/2}) - \ghat(\bz_{k+1/2})) \|^2 \notag \\
&\quad + \E\| \bz_{k+1} - \bz_{k+1/2}\|^2 - \E\| \bz_k-\bz_{k+1/2}\|^2. \label{eq: ajk4}
\end{align}
It is also easy to see by Young's inequality that 
\begin{align*}
 &\frac{\rho}{\eta_k} \| \bz_k - \bz_{k+1} + \eta_k(G(\bz_{k+1/2}) - \ghat(\bz_{k+1/2})) \|^2 \\
 & = O\left( \frac{\rho}{\eta_k}\left( 
    \| \bz_k - \bz_{k+1/2}\|^2 + \| \bz_{k+1/2} - \bz_{k+1}\|^2 + \eta_k^2\| \ghat(\bz_{k+1/2}) - G(\bz_{k+1/2})\|^2  \right) \right),
\end{align*}
and 
\begin{align*}
&\|\bz_{k+1} - \bz_{k+1/2}\|^2 \\
&= O(\eta_k^2(L^2 \|\bz_k - \bz_{k+1/2}\|^2+ \|\ghat(\bz_k)-G(\bz_k)\|^2 + \|\ghat(\bz_{k+1/2})-G(\bz_{k+1/2})\|^2)).
\end{align*}
Due to the last two bounds, with sufficiently small $\rho$, $\eta_k$ and large enough $b_k$ in \eqref{eq: ajk4}, one can cancel the error terms coming from nonmonotonicity and also coming from $\|\bz_{k+1} - \bz_{k+1/2}\|^2$.

\section{Results with MLMC}\label{sec: mlmc}
\begin{algorithm*}[t]
\caption{Inexact KM iteration (see \cite{alacaoglu2024revisiting})}
\begin{algorithmic}
    \STATE {\bfseries Input:} Parameters $\eta, N_k, M_k$, $\alpha=1-\frac{\rho}{\eta}$, $\alpha_k = \frac{\alpha}{\sqrt{k+2}\log(k+3)}$,  $\bz_0$, subroutine $\texttt{MLMC{-}FBF}$ given in Algorithm \ref{alg:fbf_mlmc} \\
    \vspace{.2cm}
    \FOR{$k = 0, 1, 2,\ldots, K-1 $}
        \STATE $\widetilde{J}^{(m)}_{\eta(G+\partial r)}(\bz_k) = \texttt{MLMC{-}FBF}\left(\bz_k, N_k, \eta \partial r, \id + \eta \gtil,1+\eta L\right)$ independently $m=1,\dots, M_k$
        \STATE $\widetilde{J}_{\eta(G+\partial r)}(\bz_k) = \frac{1}{M_k}\sum_{i=1}^{M_k} \widetilde{J}_{\eta(G+\partial r)}^{(i)}(\bz_k)$
        \STATE $\bz_{k+1} = (1-\alpha_k)\bz_k + \alpha_k \widetilde{J}_{\eta(G+\partial r)}(\bz_k)$
        \ENDFOR
      \end{algorithmic}
\label{alg:weakmvi_stoc}
\end{algorithm*}

\begin{algorithm}[!t]
\caption{$\texttt{MLMC{-}FBF}(\bx_0, N, A, B, L_B)$ (see \cite{alacaoglu2024revisiting})}
\begin{algorithmic}
    \STATE {\bfseries Input:} Initial iterate $\bx_0$, subsolver $\texttt{FBF}$ from \Cref{alg:fbf_stoc} \\
    \vspace{.2cm}
    \STATE Define $\by^i = \texttt{FBF}(\bx_0, 2^i, \widetilde{B}, A, L_B)$ for any $i\geq 0$. Draw $I\sim \mathrm{Geom}(1/2)$\\
    \vspace{.2cm}
    \STATE \textbf{Output:} $\by^{\out} = \by^0 + 2^I(\by^I - \by^{I-1})$ if $2^I \leq N$, otherwise $\by^{\text{out}} = \by^0$.
      \end{algorithmic}
\label{alg:fbf_mlmc}
\end{algorithm}
\begin{algorithm}[!t]
\caption{$\texttt{FBF}(z_0, T, A, \widetilde{B}_{\mathrm{in}}, L_B)$ from \citep{tseng2000modified} -- Stochastic}
\begin{algorithmic}
    \STATE {\bfseries Input:} Initial iterate $\bx_0$, $\widetilde{B}(\cdot) = \widetilde{B}_{\mathrm{in}}(\cdot) - \bx_0$ \\
    \vspace{.2cm}
    \FOR{$t = 0, 1, 2,\ldots, T-1 $}
        \STATE $\bx_{t+1/2} = \prox_{\tau_t r}(\bx_t - \tau_t \widetilde{B}(\bx_t)) $ 
        \STATE $\bx_{t+1} = \bx_{t+1/2} +\tau_t \widetilde{B}(\bx_t) - \tau_t \widetilde{B}(\bx_{t+1/2})$
        \ENDFOR
      \end{algorithmic}
\label{alg:fbf_stoc}
\end{algorithm}
As mentioned earlier, the parameter $\rho$ determines the level of nonmonotonicity  (or nonconvex-nonconcavity for min-max problems) in view of \eqref{eq: weakminty}. The largest-known upper bound for $\rho$ that can be tolerated by first-order methods is recently established by \citet{alacaoglu2024revisiting} who showed the best-known first-order complexity results under $\rho < 1/L$. This work only focused on analyzing their method under the bounded variance assumption. In this section, we show how to generalize their analysis under \Cref{asp: 3} which relaxes the bounded variance assumption.

\subsection{Algorithmic Ideas}
To describe the algorithmic ideas, let us recall the definition of nonexpansiveness, which asks for the operator $T\colon \mathbb{R}^m \to \mathbb{R}^m$ to satisfy $\|T\bx-T\by \| \leq \| \bx-\by\|$. 
We next recall the resolvent operator, which is a generalization of the proximal operator. In particular, we define the resolvent of $G+\partial r$ as
\begin{equation}\label{eq: def_resolvent}
J_{\eta(G+\partial r)} = (\id+\eta(G+\partial r))^{-1},
\end{equation}
where we used $\id$ to denote the identity operator.

In the special case of min-max problems, we have
\begin{align}
J_{\eta(G+\partial r)}(\bz_k) &= \arg\min_\bx \max_\by f(\bx, \by) + h_1(\bx)-h_2(\by) + \frac{1}{2\eta} \| \bx-\bx_k\|^2 - \frac{1}{2\eta} \| \by-\by_k\|^2,\label{eq: def_prox}
\end{align}
which is a strongly convex-strongly concave optimization problem. To make this connection explicit, we will refer to the estimation of approximate solutions of \eqref{eq: def_resolvent} as the \emph{proximal subproblem}.  Approximate solutions to \eqref{eq: def_resolvent} or \eqref{eq: def_prox} can be found efficiently with stochastic gradients of $f$ or stochastic evaluations of $G$ \cite{kotsalis2022simpleii}.

Then, the algorithm in \cite{alacaoglu2024revisiting} proceeds by applying the classical inexact Krasnoselskii-Mann (KM) iteration to the \emph{conically} quasi-nonexpansive operator (in view of \cite{bauschke2021generalized}) 
$J_{\eta(G+\partial r)}$,
where this property of $J_{\eta(G+\partial r)}$ under \eqref{eq: weakminty} follows from the developments in \cite{bauschke2021generalized}, see \cite{alacaoglu2024revisiting} for the details.

As motivated in \cite{alacaoglu2024revisiting}, to get the best complexity with this scheme, one needs a strong control over the bias of the estimation of $J_{\eta(G+\partial r)}$. Multilevel Monte-Carlo technique is a natural choice because it helps trade-off the bias and variance of the estimator. The primitive used in MLMC is an algorithm that can solve the strongly monotone proximal subproblem (see \eqref{eq: def_resolvent}, \eqref{eq: def_prox}) with an optimal complexity. This is indeed where bounded variance was required in the work of \cite{alacaoglu2024revisiting}. Hence we have precisely the same algorithm, but we show how it can be analyzed under more general \Cref{asp: 3}.

\subsection{Complexity Analysis}
The main complexity result in this section is summarized in the following theorem. See \Cref{th: mlmc_supp} for a detailed restatement.
\begin{theorem}\label{th: mlmc}
Let Assumptions \ref{asp: 1}, \ref{asp: 3} hold and suppose that $\rho < 1/L$. Then, for Alg. \ref{alg:weakmvi_stoc} with $\eta \leq \frac{1}{L}$ and $\alpha_k = \frac{\alpha}{\sqrt{k+2}\log(k+3)}$ (where the expressions for $N_k, M_k$ are given in App. \ref{app: proofs3}), we can generate $\bz^\out$ such that
\begin{equation*}
\E [\res(\bz^{\out})] \leq \varepsilon \text{~with expected sfo complexity~} \widetilde{O}(\varepsilon^{-4}),
\end{equation*}
where $\bz_{\hat k}$ is selected uniformly at random after running the algorithm for $K$ iterations and $\bz^{\out}$ is generated by applying one step of \eqref{eq: fbf_stoc} for problem \eqref{eq: def_prox}, starting from $\bz_{\hat k}$.
\end{theorem}
\begin{remark}
From a theoretical point-of-view, this theorem gives us the strongest result because we obtain the best-known complexity results under \eqref{eq: weakminty} with the best-known range on $\rho$ without bounded variance. On the other hand, a theoretical limitation of this approach is that the number of stochastic oracle calls at each iteration is random, causing the final complexity result to be on the \emph{expected} number of oracle calls (which is shared by all approaches relying on MLMC). A practical drawback is that the algorithm is not single-loop.
\end{remark}
We start with a lemma from \cite{alacaoglu2024revisiting} which only uses \Cref{asp: 1} and gives a bound on how close $\bz_k$ is to be a solution of \eqref{eq: inc_prob} since a fixed-point of the resolvent $J_{\eta(G+\partial r)}$ is a solution of \eqref{eq: inc_prob}.

\begin{lemma}(See \citep[Lemma C.8]{alacaoglu2024revisiting})\label{eq: lem_resid}
Let \Cref{asp: 1} hold. We then have
\begin{align*}
& \sum_{k=0}^{K-1} \alpha_k \E \| (\id - J_{\eta(G+\partial r)})(\bz_k)\|^2 \\&= O\Big( \|\bz_0 - \bz^\star\|^2 \\
&\quad+  \E\sum_{k=0}^{K-1} \Big[\alpha_k^2 \|J_{\eta(G+\partial r)}(\bz_k) - \widetilde{J}_{\eta(G+\partial r)}(\bz_k)\|^2 + \alpha_k \|\bz_k-\bz^\star\| \|J_{\eta(G+\partial r)}(\bz_k) - \E_k[\widetilde{J}_{\eta(G+\partial r)}(\bz_k)]\|\Big]\Big).
\end{align*}
\end{lemma}

The main message of this lemma is that the error in form \emph{bias}, that is $ \|J_{\eta(G+\partial r)}(\bz_k) - \E_k[\widetilde{J}_{\eta(G+\partial r)}(\bz_k)]\|$ is only multiplied by $\alpha_k$ whereas the error in form \emph{variance} $\E \|J_{\eta(G+\partial r)}(\bz_k) - \widetilde{J}_{\eta(G+\partial r)}(\bz_k)\|^2$ is multiplied by $\alpha_k^2$. Since $\alpha_k$ is small, this means that the analysis has the former error as the bottleneck.

This is where MLMC comes into play and the primitive in this estimator is a subsolver for the proximal subproblem. As mentioned before, this is the only place where \cite{alacaoglu2024revisiting} needed the bounded variance assumption. We next show that the bounded variance is in fact not needed: by a slightly different choice of step size, one can incorporate \Cref{asp: 3}. Similar results appeared in \cite{kotsalis2022simpleii,wang2015incremental} for different algorithms. We provide our analysis for simplicity and to be self-contained.
\begin{lemma}\label{eq: siv4}
Let Assumptions \ref{asp: 1}, \ref{asp: 3} hold. Let $\tau_t = \Theta\left( \frac{1}{t\mu+L^2/\mu} \right)$. We have for the output of Alg. \ref{alg:fbf_stoc} that
\begin{equation*}
\E \| \bx_T - J_{\eta(G+\partial r)}(\bz_k) \|^2 = O(T^{-1}).
\end{equation*}
\end{lemma}
This lemma tells us how close $\bx_T$ is to being a solution of the proximal subproblem, see \eqref{eq: def_resolvent} and \eqref{eq: def_prox}. The main idea is extremely simple and can be found in standard textbooks for the minimization case, see \citep[Section 5.4.3]{wright2022optimization}. In particular, once we have strong convexity (or strong convexity-strong concavity), handling the additional error terms coming from \Cref{asp: 3} is straightforward, because the negative term that the strong convexity gives can be used to cancel these. 

Let us provide the sketch of the argument here. Without the explicit constants, the main recursion for the stochastic FBF under bounded variance is the following
\begin{equation*}
\E\| \bz_{t+1} - \bz^\star \|^2 \leq (1-\tau_t\mu)\E\| \bz_{t} - \bz^\star \|^2 + O(\tau_t^2),
\end{equation*}
which one can use to get a rate $O(1/t)$, by induction. Under \Cref{asp: 3}, the recursion becomes
\begin{equation*}
\E\| \bz_{t+1} - \bz^\star \|^2 \leq (1-\tau_t\mu+\tau_t^2B^2)\E\| \bz_{t} - \bz^\star \|^2 + O(\tau_t^2).
\end{equation*}
Focusing on the coefficient of the first term on the right-hand side, the new error term scales as $\tau_t^2$. This can be absorbed in the term $-\tau_t\mu$ since for small $\tau_t$, we have $\tau_t^2 < \tau_t$. Hence, even though our problem is, strictly speaking, more general, the resulting recursion is the same as SGD and the same ideas as \citep[Section 5.4.3]{wright2022optimization} can be used.

Equipped with this result, which gives us the desired behavior from the subsolver under \Cref{asp: 3}, the analysis follows the same steps as \cite{alacaoglu2024revisiting}.

\section{Variance Reduction}\label{sec: vr}
\begin{algorithm}[H]
   \caption{Variance reduced FBF with Halpern anchoring}
   \label{alg:example}
\begin{algorithmic}
   \STATE {\bfseries Input:} Initial iterate $\bz_0$ and parameter $\rho \geq 0$.
   \FOR{$k=0, \dots$}
   \STATE $\bar \bz_k = \beta_k \bz_0 + (1-\beta_k) \bz_k$
   \STATE $\bz_{k+1/2} = \prox_{\gamma_k r}(\bar\bz_k - \gamma_k \bg_k )$
   \STATE Set $\widetilde{G}_{k+1/2} = \widetilde G(\bz_{k+1/2}, \xi_{k+1/2})$
   \STATE $\bz_{k+1} = \bar\bz_k - \tau_k (\bar \bz_k - \bz_{k+1/2}- \gamma_k \bg_k + \gamma_k \widetilde{G}_{k+1/2})$
   \STATE $\bg_{k+1} = \widetilde{G}(\bz_{k+1}, \xi_{k+1}) + (1-\alpha_k)(\bg_{k} - \widetilde{G}(\bz_{k}, \xi_{k+1}))$
   \ENDFOR
\end{algorithmic}
\label{alg: var_red_fbf}
\end{algorithm}

In this section, we design a single-loop algorithm that does not rely on large mini-batch sizes, that only uses $3$ unbiased samples of $G$ at every iteration. 

\subsection{Algorithmic Ideas}
The three main components of this algorithm are  \emph{(i)} variance reduction via STORM estimator \cite{cutkosky2019momentum}, \emph{(ii)} the FBF method \cite{tseng2000modified} and \emph{(iii)} the Halpern anchoring \citep{halpern1967fixed,yoon2021accelerated}. This algorithm and our analysis in this section are combining the ideas from \citep{pethick2023solving, alacaoglu2025towards} to combine the best-of-both-worlds in each case. Compared to the former work, we show guarantees without bounded variance and compared to the latter, our results can accommodate $\rho >0$, all with the same complexity. 
An earlier work on the FBF variant used in Alg. \ref{alg: var_red_fbf} is \citep{giselsson2021nonlinear}.

In particular, when $\beta_k\equiv 0$ in Alg. \ref{alg: var_red_fbf}, one can see that this algorithm reduces to the one in \cite{pethick2023solving}. When $\tau = 1$, one can also notice the similarity between this algorithm and the classical FBF algorithm of \cite{tseng2000modified}. Indeed, FBF is the same as Alg. \ref{alg: var_red_fbf} when $\beta_k\equiv 0, \tau=1$ and when the estimators are replaced with the full gradients.

As recently observed by \cite{neu2024dealing} for convex-concave min-max problems, including a nonzero $\beta_k$ is important to handle \Cref{asp: 3}. With the selection of $\beta_k = \Theta(1/k)$, the anchoring in the algorithm corresponds to that of the classical Halpern iteration \citep{halpern1967fixed,yoon2021accelerated}.

\subsection{Complexity Analysis}
We state the main result and then the main proof ideas. The complete proof is rather technical and is provided in App. \ref{app: proofs4}. See \Cref{th: vr_supp} for a precise statement and Remark \ref{eq: svx3} for further discussion about the parameters.
\begin{theorem}\label{th: vr}
Let Assumptions \ref{asp: 1}, \ref{asp: 2}, \ref{asp: 3} hold and let $\rho \leq f(\bar\tau)$ for a function $f$ where $f(\bar\tau) \to 1/(16L)$ as $\bar\tau \to 0$ (the precise form of $f(\bar\tau)$ appears in \Cref{th: vr_supp}). Then, for the output of Alg. \ref{alg: var_red_fbf} with $\beta_k = \Theta(1/k)$, $\gamma_k=\Theta(1/L_{\mathrm{exp}})$, $\tau_k = \Theta(1/\sqrt{k})$, we have
\begin{equation*}
\E [\res(\bz^{\out})] \leq \varepsilon \text{~with sfo complexity~} \widetilde{O}(\varepsilon^{-4}),
\end{equation*}
where $\Pr(\hat k=k) = \frac{\tau_k(k+3)}{\sum_{i=0}^{K-1} \tau_i(i+3)}$ and $\bz^{\out} = \bz_{\hat k+1/2}$.
\end{theorem}
\begin{remark}
This result gives us the simplest algorithmic construction compared to earlier sections. We neither need large mini-batch sizes nor inner loops in this method. The cost is the need for the slightly stronger oracle and the Lipschitzness assumption given in \Cref{asp: 2}, which, in fact, is common for variance reduction \cite{arjevani2023lower,pethick2023solving}. Another limitation compared to Sec. \ref{sec: mlmc} is that  the upper bound for $\rho$ is suboptimal (which is not optimized). 
\end{remark}

\paragraph{Proof sketch. } One critical property of this FBF-based method in \Cref{alg: var_red_fbf} (compared to the algorithm considered in \cite{alacaoglu2025towards}) is the following property (similar to \Cref{subsec: mb_complex}):
Notice that $\frac{1}{\gamma_k}(\bar\bz_k - \bz_{k+1/2}) -  \bg_k + G(\bz_{k+1/2})\in (G+\partial r)(\bz_{k+1/2})$ and hence using \eqref{eq: weakminty}:
\begin{align*}
\gamma_k^{-1}\langle \bar\bz_k - \bz_{k+1/2} - \gamma_k \bg_k, \bz_{k+1/2} - \bz^\star \rangle &+ \langle G(\bz_{k+1/2}), \bz_{k+1/2} - \bz^\star \rangle \\
&\quad\geq -\rho \| \gamma_k^{-1}(\bar\bz_k - \bz_{k+1/2}) - \bg_k +  G(\bz_{k+1/2}) \|^2.
\end{align*}
On the other hand, analyzing the inner products on the left-hand side gives a recursion of the form
\begin{align*}
\E \| \bz^\star -\bz_{k+1}\|^2 \leq (1-\beta_k)\E \| \bz^\star -\bz_{k}\|^2  + \beta_k \| \bz^\star - \bz_0\|^2 -\good_k + \bad_k + O(\tau_k^2)
\end{align*}
where
\begin{align*}
    \good_k &= \Theta(\tau_k)\E \| \bz_k-\bz_{k+1/2}\|^2 + \Theta(\beta_k) \E\|\bz_0-\bz_{k+1}\|^2  + \Theta(\tau_k\beta_k) \E \| \bz_0-\bz_{k+1/2}\|^2 + \Theta(1)\E\|\bz_k-\bz_{k+1}\|^2
    \end{align*}
    and
    \begin{align*}
    \;\;\bad_k  &= \Theta(\tau_k)\E\|\bg_k - G(\bz_k)\|^2 + \Theta\left(\frac{\rho\tau_k}{\gamma_k}\right) \E \| \bar\bz_k - \bz_{k+1/2} - \gamma_k(\bg_k - G(\bz_{k+1/2}))\|^2.
\end{align*}
On a high level, one can see 
\begin{align*}&\frac{\rho\tau_k}{\gamma_k} \E \| \bar\bz_k - \bz_{k+1/2} - \gamma_k(\bg_k - G(\bz_{k+1/2}))\|^2 \\
&\leq \Theta\left( \frac{\rho\tau_k}{\gamma_k} +\rho\tau_k\gamma_kL^2 \right) \E \| \bz_k-\bz_{k+1/2}\|^2 + \Theta\left( \frac{\rho\tau_k\beta_k}{\gamma_k} \right) \E \| \bz_0-\bz_{k+1/2}\|^2 + \Theta\left( \rho\tau_k\gamma_k \right) \E \| \bg_k - G(\bz_k)\|^2
\end{align*}
By selecting the parameters accordingly, when $\rho$ is small enough (see Thm. \ref{th: vr}) and the orders of the parameters are $\tau_k=\Theta(1/\sqrt{k})$, $\beta_k = \Theta(1/k)$, $\gamma_k = \Theta(1/L)$, the error terms can be cancelled.

The final error will be due to the variance error $\mathbb{E}\|\bg_k - G(\bz_k)\|^2$. The control for this comes from the standard bounds on STORM estimator \citep{cutkosky2019momentum}, see Thm. \ref{lem: vr_storm_lem}. Routine calculations help us finish the proof.

\section{Numerical Results}\label{sec: numer}
Our numerical experiments evaluate robustness across three benchmarks. (\emph{Left}) We reproduce the counter-example from \citep[Thm. 4.3]{gorbunov2023convergence} with operator $F(x) = LAx$ where $A$ is a rotation matrix. We set $\rho=\tfrac{1}{2L}$, a regime in which EG is known to lose stability on this instance.
In this setting, EG diverges as expected, whereas Alg. \ref{alg:weakmvi_stoc} converges. (\emph{Middle/Right}) Here, we compare Alg. \ref{alg: var_red_fbf} with \citep[Alg. 1]{pethick2023solving} on the unconstrained quadratic Example~2 of \cite{pethick2023solving} (see \eqref{eq: uq_prob}) with noisy oracle: the middle figure uses noise with Student's $t$-distribution with parameter $\nu=2$ and the right figure uses Laplace noise. For each step size $\gamma$ on a fixed grid, we run 7 independent seeds per method and plot the mean of the results. As $\gamma$ decreases, Alg. \ref{alg: var_red_fbf} remains stable and convergent over a broad range, whereas \citep[Alg. 1]{pethick2023solving} diverges for small $\gamma$, demonstrating stronger robustness of Alg. \ref{alg: var_red_fbf} to noisy oracles. See App. \ref{app: exper} for details.

\begin{figure*}[ht!]
  \centering
  \subfigure{%
    \includegraphics[width=0.32\linewidth]{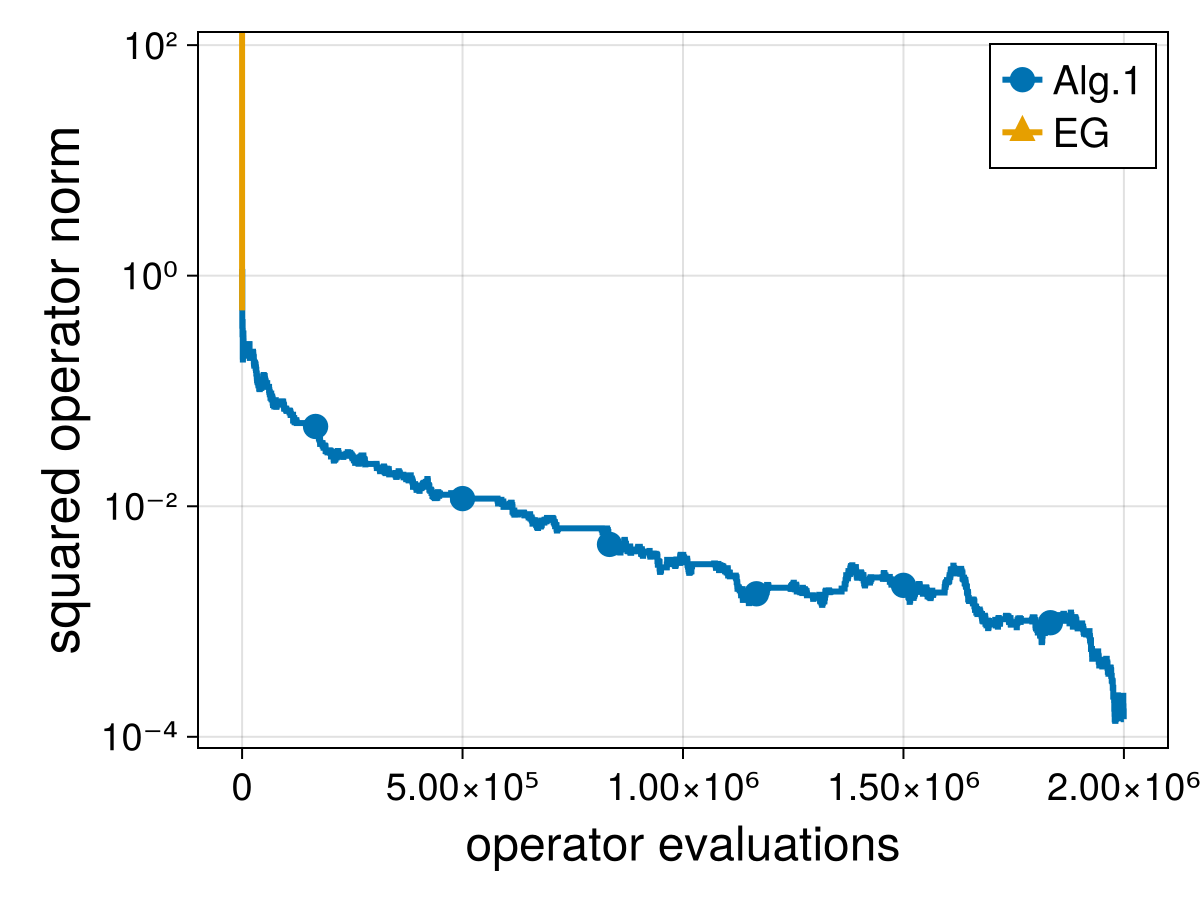}
  }\hfill
  \subfigure{%
    \includegraphics[width=0.32\linewidth]{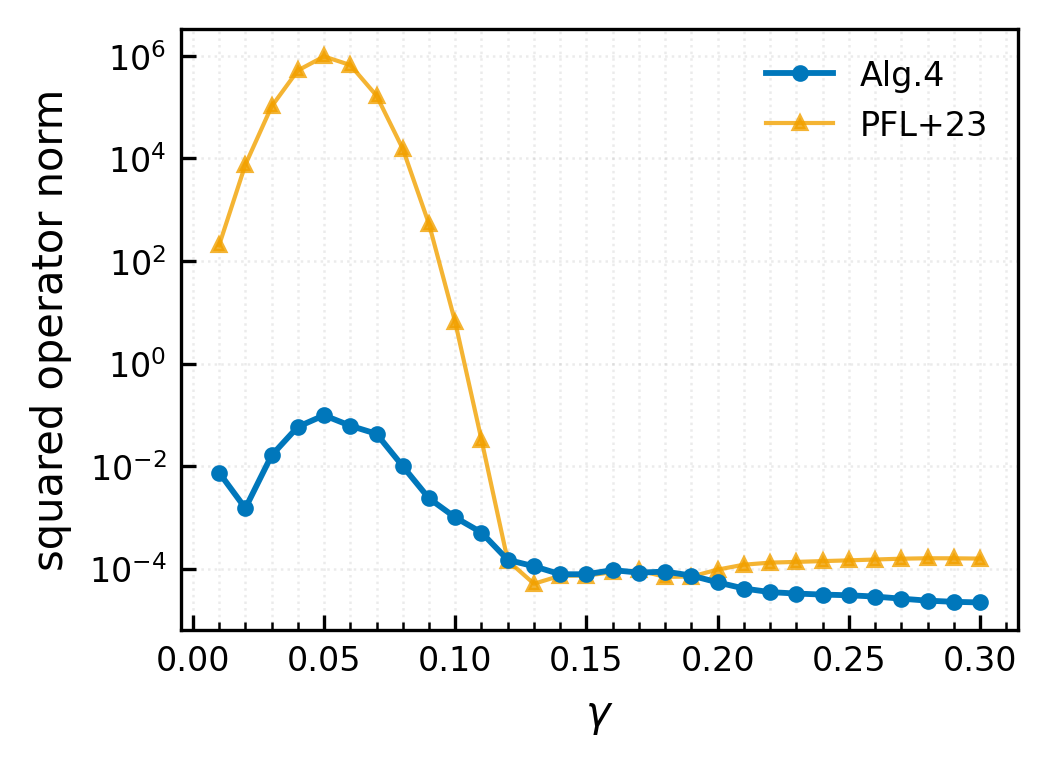}
  }\hfill
  \subfigure{%
    \includegraphics[width=0.32\linewidth]{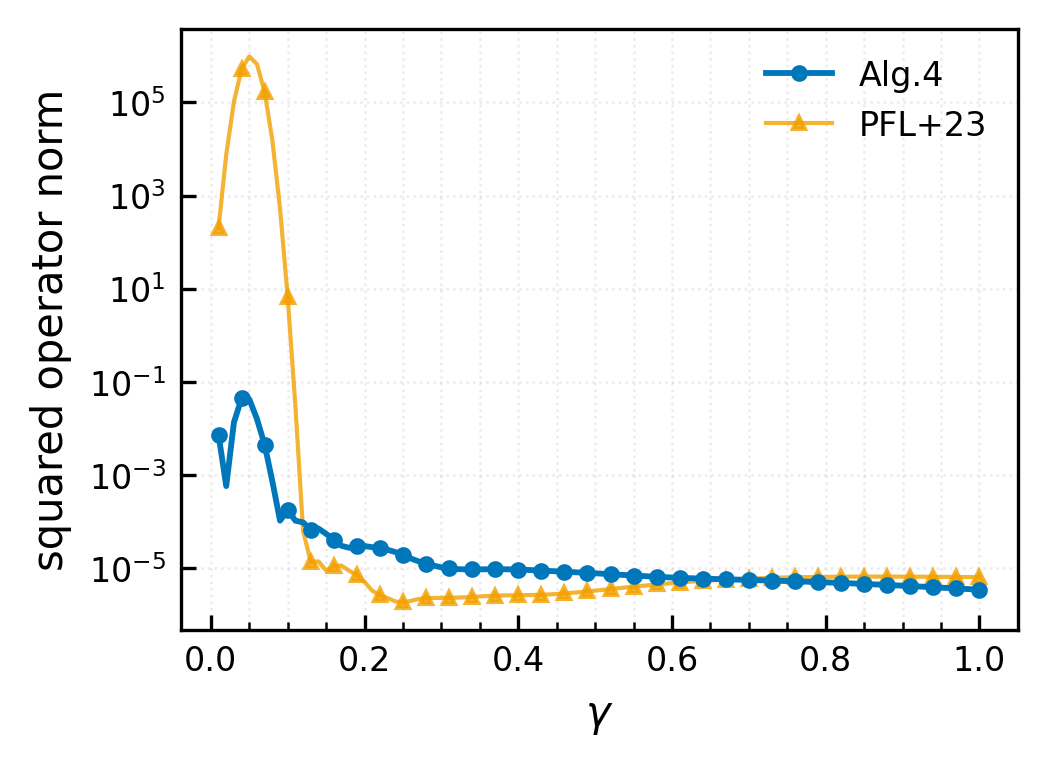}
  }
  \caption{\small{Left: Alg. \ref{alg:weakmvi_stoc} and EG for the counter-example problem (see \eqref{eq: def_rot}). Middle and right: Alg. \ref{alg: var_red_fbf} and the algorithm of \cite{pethick2023solving} for the unconstrained problem in \eqref{eq: uq_prob} with operator noise distributed as Student's t and Laplace-distribution. 
  } 
  } 
  \label{fig:lax-gamma-png}
\end{figure*}

\section{Conclusions}
We provided three different algorithms that all achieve complexity $\widetilde{O}(\varepsilon^{-4})$ for solving nonmonotone stochastic variational inequalities without uniformly bounded variance (which is the best-known complexity, up-to $\log$ terms, even for constrained monotone SVIs with bounded variance). These algorithms provide distinct advantages compared to each other and hence are complementary. An open question is to design and analyze a unified algorithm that can achieve the \emph{best-of-three-worlds}: a single-loop algorithm with the best range for $\rho$ and without large mini-batches. Moreover, for problems satisfying \eqref{eq: weakminty}, it is still unknown whether it is possible to improve $O(\varepsilon^{-4})$ for getting $\mathbb{E}[\res(\bz)]\leq \varepsilon$, even with bounded variance and no constraints.

\section*{Acknowledgments}
Ahmet Alacaoglu acknowledges the support of the Natural Sciences and Engineering Research Council of Canada (NSERC), [funding reference number RGPIN-2025-06634].

\bibliographystyle{plainnat}
\bibliography{lit}

\begin{thebibliography}{44}
\providecommand{\natexlab}[1]{#1}
\providecommand{\url}[1]{\texttt{#1}}
\expandafter\ifx\csname urlstyle\endcsname\relax
  \providecommand{\doi}[1]{doi: #1}\else
  \providecommand{\doi}{doi: \begingroup \urlstyle{rm}\Url}\fi

\bibitem[Alacaoglu et~al.(2023)Alacaoglu, B{\"o}hm, and
  Malitsky]{alacaoglu2023beyond}
Ahmet Alacaoglu, Axel B{\"o}hm, and Yura Malitsky.
\newblock Beyond the golden ratio for variational inequality algorithms.
\newblock \emph{Journal of Machine Learning Research}, 24\penalty0
  (172):\penalty0 1--33, 2023.

\bibitem[Alacaoglu et~al.(2024)Alacaoglu, Kim, and
  Wright]{alacaoglu2024revisiting}
Ahmet Alacaoglu, Donghwan Kim, and Stephen Wright.
\newblock Revisiting inexact fixed-point iterations for min-max problems:
  Stochasticity and structured nonconvexity.
\newblock In \emph{International Conference on Machine Learning}, pages
  840--878. PMLR, 2024.

\bibitem[Alacaoglu et~al.(2025)Alacaoglu, Malitsky, and
  Wright]{alacaoglu2025towards}
Ahmet Alacaoglu, Yura Malitsky, and Stephen~J Wright.
\newblock Towards weaker variance assumptions for stochastic optimization.
\newblock \emph{arXiv:2504.09951}, 2025.

\bibitem[Arjevani et~al.(2023)Arjevani, Carmon, Duchi, Foster, Srebro, and
  Woodworth]{arjevani2023lower}
Yossi Arjevani, Yair Carmon, John~C Duchi, Dylan~J Foster, Nathan Srebro, and
  Blake Woodworth.
\newblock Lower bounds for non-convex stochastic optimization.
\newblock \emph{Mathematical Programming}, 199\penalty0 (1-2):\penalty0
  165--214, 2023.

\bibitem[Asi et~al.(2021)Asi, Carmon, Jambulapati, Jin, and
  Sidford]{asi2021stochastic}
Hilal Asi, Yair Carmon, Arun Jambulapati, Yujia Jin, and Aaron Sidford.
\newblock Stochastic bias-reduced gradient methods.
\newblock \emph{Advances in Neural Information Processing Systems},
  34:\penalty0 10810--10822, 2021.

\bibitem[Bauschke and Combettes(2017)]{bauschke2017convex}
Heinz~H Bauschke and Patrick~L Combettes.
\newblock Convex analysis and monotone operator theory in hilbert spaces.
\newblock \emph{CMS Books in Mathematics}, 2017.

\bibitem[Bauschke et~al.(2021)Bauschke, Moursi, and
  Wang]{bauschke2021generalized}
Heinz~H Bauschke, Walaa~M Moursi, and Xianfu Wang.
\newblock Generalized monotone operators and their averaged resolvents.
\newblock \emph{Mathematical Programming}, 189:\penalty0 55--74, 2021.

\bibitem[Bertsekas(2014)]{bertsekas2014constrained}
Dimitri~P Bertsekas.
\newblock \emph{Constrained optimization and Lagrange multiplier methods}.
\newblock Academic press, 2014.

\bibitem[B{\"o}hm et~al.(2022)B{\"o}hm, Sedlmayer, Csetnek, and
  Bot]{bohm2022two}
Axel B{\"o}hm, Michael Sedlmayer, Erno~Robert Csetnek, and Radu~Ioan Bot.
\newblock Two steps at a time---taking gan training in stride with tseng's
  method.
\newblock \emph{SIAM Journal on Mathematics of Data Science}, 4\penalty0
  (2):\penalty0 750--771, 2022.

\bibitem[Choudhury et~al.(2023)Choudhury, Gorbunov, and
  Loizou]{choudhury2023single}
Sayantan Choudhury, Eduard Gorbunov, and Nicolas Loizou.
\newblock Single-call stochastic extragradient methods for structured
  non-monotone variational inequalities: Improved analysis under weaker
  conditions.
\newblock \emph{Advances in Neural Information Processing Systems},
  36:\penalty0 64918--64956, 2023.

\bibitem[Combettes and Pennanen(2004)]{combettes2004proximal}
Patrick~L Combettes and Teemu Pennanen.
\newblock Proximal methods for cohypomonotone operators.
\newblock \emph{SIAM journal on control and optimization}, 43\penalty0
  (2):\penalty0 731--742, 2004.

\bibitem[Cutkosky and Orabona(2019)]{cutkosky2019momentum}
Ashok Cutkosky and Francesco Orabona.
\newblock Momentum-based variance reduction in non-convex {SGD}.
\newblock \emph{Advances in neural information processing systems}, 32, 2019.

\bibitem[Diakonikolas et~al.(2021)Diakonikolas, Daskalakis, and
  Jordan]{diakonikolas2021efficient}
Jelena Diakonikolas, Constantinos Daskalakis, and Michael~I Jordan.
\newblock Efficient methods for structured nonconvex-nonconcave min-max
  optimization.
\newblock In \emph{International Conference on Artificial Intelligence and
  Statistics}, pages 2746--2754. PMLR, 2021.

\bibitem[Duchi and Namkoong(2021)]{duchi2021learning}
John~C Duchi and Hongseok Namkoong.
\newblock Learning models with uniform performance via distributionally robust
  optimization.
\newblock \emph{The Annals of Statistics}, 49\penalty0 (3):\penalty0
  1378--1406, 2021.

\bibitem[Facchinei and Pang(2003)]{facchinei2003finite}
Francisco Facchinei and Jong-Shi Pang.
\newblock \emph{Finite-dimensional variational inequalities and complementarity
  problems}.
\newblock Springer, 2003.

\bibitem[Giselsson(2021)]{giselsson2021nonlinear}
Pontus Giselsson.
\newblock Nonlinear forward-backward splitting with projection correction.
\newblock \emph{SIAM Journal on Optimization}, 31\penalty0 (3):\penalty0
  2199--2226, 2021.

\bibitem[Gladyshev(1965)]{gladyshev1965stochastic}
EG~Gladyshev.
\newblock On stochastic approximation.
\newblock \emph{Theory of Probability \& Its Applications}, 10\penalty0
  (2):\penalty0 275--278, 1965.

\bibitem[Goodfellow et~al.(2014)Goodfellow, Pouget-Abadie, Mirza, Xu,
  Warde-Farley, Ozair, Courville, and Bengio]{goodfellow2014generative}
Ian~J Goodfellow, Jean Pouget-Abadie, Mehdi Mirza, Bing Xu, David Warde-Farley,
  Sherjil Ozair, Aaron Courville, and Yoshua Bengio.
\newblock Generative adversarial nets.
\newblock \emph{Advances in neural information processing systems}, 27, 2014.

\bibitem[Gorbunov et~al.(2023)Gorbunov, Taylor, Horv{\'a}th, and
  Gidel]{gorbunov2023convergence}
Eduard Gorbunov, Adrien Taylor, Samuel Horv{\'a}th, and Gauthier Gidel.
\newblock Convergence of proximal point and extragradient-based methods beyond
  monotonicity: the case of negative comonotonicity.
\newblock In \emph{International Conference on Machine Learning}, pages
  11614--11641. PMLR, 2023.

\bibitem[Grimmer et~al.(2023)Grimmer, Lu, Worah, and
  Mirrokni]{grimmer2023landscape}
Benjamin Grimmer, Haihao Lu, Pratik Worah, and Vahab Mirrokni.
\newblock The landscape of the proximal point method for nonconvex--nonconcave
  minimax optimization.
\newblock \emph{Mathematical Programming}, 201\penalty0 (1-2):\penalty0
  373--407, 2023.

\bibitem[Halpern(1967)]{halpern1967fixed}
Benjamin Halpern.
\newblock Fixed points of nonexpanding maps.
\newblock \emph{Bulletin of the American Mathematical Society}, 73\penalty0
  (6):\penalty0 957--961, 1967.

\bibitem[Iusem et~al.(2017)Iusem, Jofre, Oliveira, and
  Thompson]{iusem2017extragradient}
AN~Iusem, A~Jofre, RI~Oliveira, and P~Thompson.
\newblock Extragradient method with variance reduction for stochastic
  variational inequalities.
\newblock \emph{SIAM Journal on Optimization}, 27\penalty0 (2):\penalty0
  686--724, 2017.

\bibitem[Khaled and Richt{\'a}rik(2023)]{khaled2022better}
Ahmed Khaled and Peter Richt{\'a}rik.
\newblock Better theory for {SGD} in the nonconvex world.
\newblock \emph{Transactions on Machine Learning Research}, 2023.

\bibitem[Khaled et~al.(2023)Khaled, Sebbouh, Loizou, Gower, and
  Richt{\'a}rik]{khaled2023unified}
Ahmed Khaled, Othmane Sebbouh, Nicolas Loizou, Robert~M Gower, and Peter
  Richt{\'a}rik.
\newblock Unified analysis of stochastic gradient methods for composite convex
  and smooth optimization.
\newblock \emph{Journal of Optimization Theory and Applications}, 199\penalty0
  (2):\penalty0 499--540, 2023.

\bibitem[Korpelevich(1976)]{korpelevich1976extragradient}
Galina~M Korpelevich.
\newblock The extragradient method for finding saddle points and other
  problems.
\newblock \emph{Matecon}, 12:\penalty0 747--756, 1976.

\bibitem[Kotsalis et~al.(2022)Kotsalis, Lan, and Li]{kotsalis2022simpleii}
Georgios Kotsalis, Guanghui Lan, and Tianjiao Li.
\newblock Simple and optimal methods for stochastic variational inequalities,
  ii: Markovian noise and policy evaluation in reinforcement learning.
\newblock \emph{SIAM Journal on Optimization}, 32\penalty0 (2):\penalty0
  1120--1155, 2022.

\bibitem[Lan(2020)]{lan2020first}
Guanghui Lan.
\newblock \emph{First-order and stochastic optimization methods for machine
  learning}, volume~1.
\newblock Springer, 2020.

\bibitem[Lan(2023)]{lan2023policy}
Guanghui Lan.
\newblock Policy mirror descent for reinforcement learning: Linear convergence,
  new sampling complexity, and generalized problem classes.
\newblock \emph{Mathematical programming}, 198\penalty0 (1):\penalty0
  1059--1106, 2023.

\bibitem[Lee and Kim(2021)]{lee2021fast}
Sucheol Lee and Donghwan Kim.
\newblock Fast extra gradient methods for smooth structured
  nonconvex-nonconcave minimax problems.
\newblock \emph{Advances in Neural Information Processing Systems},
  34:\penalty0 22588--22600, 2021.

\bibitem[Li et~al.(2025)Li, Zhu, and So]{li2025nonsmooth}
Jiajin Li, Linglingzhi Zhu, and Anthony Man-Cho So.
\newblock Nonsmooth nonconvex--nonconcave minimax optimization: Primal--dual
  balancing and iteration complexity analysis.
\newblock \emph{Mathematical Programming}, pages 1--51, 2025.

\bibitem[Madry et~al.(2018)Madry, Makelov, Schmidt, Tsipras, and
  Vladu]{madry2018towards}
Aleksander Madry, Aleksandar Makelov, Ludwig Schmidt, Dimitris Tsipras, and
  Adrian Vladu.
\newblock Towards deep learning models resistant to adversarial attacks.
\newblock In \emph{International Conference on Learning Representations}, 2018.

\bibitem[Mertikopoulos(2019)]{mertikopoulos:tel-02428077}
Panayotis Mertikopoulos.
\newblock \emph{{Online optimization and learning in games: Theory and
  applications}}.
\newblock Habilitation {\`a} diriger des recherches, {Grenoble 1 UGA -
  Universit{\'e} Grenoble Alpes}, 2019.
\newblock URL \url{https://inria.hal.science/tel-02428077}.

\bibitem[Mertikopoulos et~al.(2019)Mertikopoulos, Lecouat, Zenati, Foo,
  Chandrasekhar, and Piliouras]{mertikopoulos2019optimistic}
Panayotis Mertikopoulos, Bruno Lecouat, Houssam Zenati, Chuan-Sheng Foo, Vijay
  Chandrasekhar, and Georgios Piliouras.
\newblock Optimistic mirror descent in saddle-point problems: Going the extra
  (gradient) mile.
\newblock In \emph{ICLR 2019-7th International Conference on Learning
  Representations}, pages 1--23, 2019.

\bibitem[Nemirovski et~al.(2009)Nemirovski, Juditsky, Lan, and
  Shapiro]{nemirovski2009robust}
Arkadi Nemirovski, Anatoli Juditsky, Guanghui Lan, and Alexander Shapiro.
\newblock Robust stochastic approximation approach to stochastic programming.
\newblock \emph{SIAM Journal on optimization}, 19\penalty0 (4):\penalty0
  1574--1609, 2009.

\bibitem[Nesterov(2007)]{nesterov2007dual}
Yurii Nesterov.
\newblock Dual extrapolation and its applications to solving variational
  inequalities and related problems.
\newblock \emph{Mathematical Programming}, 109\penalty0 (2):\penalty0 319--344,
  2007.

\bibitem[Neu and Okolo(2024)]{neu2024dealing}
Gergely Neu and Nneka Okolo.
\newblock Dealing with unbounded gradients in stochastic saddle-point
  optimization.
\newblock In \emph{Proceedings of the 41st International Conference on Machine
  Learning}, pages 37508--37530, 2024.

\bibitem[Pethick et~al.(2022)Pethick, Latafat, Patrinos, Fercoq, and
  Cevher]{pethick2022escaping}
Thomas Pethick, Puya Latafat, Panos Patrinos, Olivier Fercoq, and Volkan
  Cevher.
\newblock Escaping limit cycles: Global convergence for constrained
  nonconvex-nonconcave minimax problems.
\newblock In \emph{International Conference on Learning Representations}, 2022.

\bibitem[Pethick et~al.(2023)Pethick, Fercoq, Latafat, Patrinos, and
  Cevher]{pethick2023solving}
Thomas Pethick, Olivier Fercoq, Puya Latafat, Panagiotis Patrinos, and Volkan
  Cevher.
\newblock Solving stochastic weak minty variational inequalities without
  increasing batch size.
\newblock In \emph{International Conference on Learning Representations}, 2023.

\bibitem[Robbins and Monro(1951)]{robbins1951stochastic}
Herbert Robbins and Sutton Monro.
\newblock A stochastic approximation method.
\newblock \emph{The annals of mathematical statistics}, pages 400--407, 1951.

\bibitem[Rockafellar(1976)]{rockafellar1976monotone}
R~Tyrrell Rockafellar.
\newblock Monotone operators and the proximal point algorithm.
\newblock \emph{SIAM journal on control and optimization}, 14\penalty0
  (5):\penalty0 877--898, 1976.

\bibitem[Tseng(2000)]{tseng2000modified}
Paul Tseng.
\newblock A modified forward-backward splitting method for maximal monotone
  mappings.
\newblock \emph{SIAM Journal on Control and Optimization}, 38\penalty0
  (2):\penalty0 431--446, 2000.

\bibitem[Wang and Bertsekas(2015)]{wang2015incremental}
Mengdi Wang and Dimitri~P Bertsekas.
\newblock Incremental constraint projection methods for variational
  inequalities.
\newblock \emph{Mathematical Programming}, 150:\penalty0 321--363, 2015.

\bibitem[Wright and Recht(2022)]{wright2022optimization}
Stephen~J Wright and Benjamin Recht.
\newblock \emph{Optimization for data analysis}.
\newblock Cambridge University Press, 2022.

\bibitem[Yoon and Ryu(2021)]{yoon2021accelerated}
TaeHo Yoon and Ernest~K Ryu.
\newblock Accelerated algorithms for smooth convex-concave minimax problems
  with ${O} (1/k^2)$ rate on squared gradient norm.
\newblock In \emph{International Conference on Machine Learning}, pages
  12098--12109. PMLR, 2021.

\end{thebibliography}

\newpage
\appendix

\onecolumn
\section{Preliminaries}\label{app: prelim}
\paragraph{Notation. } In the sequel, we use the notation $\mathbb{E}_k$ to denote the conditional expectation where we condition on the $\sigma$-algebra generated by all the iterates until $\bz_k$. The notation $\mathbb{E}_{k+1/2}$ is defined similarly. 
\begin{fact}\label{lem: bg_equiv}
    The inequalities \eqref{eq: bg_def} and \eqref{eq: bg_def2} are equivalent up to a redefinition of constants.
\end{fact}
\begin{proof}
    Using $\|a+b\|^2 \leq 2\|a\|^2+2\|b\|^2$ and $\bz-\bz^\star = (\bz-\bz_0) + (\bz_0-\bz^\star)$, we have
    \begin{align*}
        \|\bz-\bz^\star\|^2 \leq 2\|\bz-\bz_0\|^2 + 2\|\bz_0-\bz^\star\|^2.
    \end{align*}
    Substituting this into \eqref{eq: bg_def} gives
    \begin{align*}
        \E \| \gtil(\bz) - G(\bz)\|^2 \leq 2B^2 \| \bz-\bz_0 \|^2 + (2B^2 \| \bz_0-\bz^\star \|^2 +\sigma^2),
    \end{align*}
    which matches \eqref{eq: bg_def2} after renaming constants since $\|\bz_0-\bz^\star\|^2$ is a constant.
    
    The reverse direction follows similarly by using $\bz-\bz_0 = (\bz-\bz^\star) + (\bz^\star-\bz_0)$.
\end{proof}

\section{Proofs for Section \ref{sec: mb_main}}\label{app: proofs2}
We start by analyzing one iteration of the algorithm given in \Cref{eq: fbf_stoc}. This is a rather standard analysis in the monotone case, see for example \cite{bohm2022two}. What we have in this lemma is a straightforward generalization that includes the \eqref{eq: weakminty} assumption instead of monotonicity.
\begin{lemma}\label{lem: app1}
Let \Cref{asp: 1} hold. Then we have
\begin{align*}
\E \| \bz^\star - \bz_{k+1} \|^2 &\leq \E \|\bz^\star - \bz_k\|^2 + \left( \frac{2(1+1/\sqrt{6})\rho}{\eta_k}+ \frac{\eta_k L^2}{29}(30\eta_k+62(1+\sqrt{6})\rho) - 1 \right) \E \| \bz_k-\bz_{k+1/2}\|^2 \\
&\quad+ \left( 186(1+\sqrt{6})\rho\eta_k + 60\eta_k^2 \right)\E \| \ghat(\bz_{k+1/2}) - G(\bz_{k+1/2})\|^2 \\
&\quad + \left( 124(1+\sqrt{6})\rho\eta_k + 60\eta_k^2 \right) \E \| \ghat(\bz_{k}) - G(\bz_{k})\|^2.
\end{align*}
\end{lemma}
\begin{proof}
The definitions of $\bz_{k+1/2}$ and the proximal operator give us that
\begin{align*}
&\bz_{k+1/2} = \arg\min_{\bz} r(\bz) + \frac{1}{2\eta_k} \| \bz-(\bz_k-\eta_k \ghat(\bz_k)) \|^2 \\
 \iff & \frac{1}{\eta_k}(\bz_{k+1/2} - \bz_k) + \ghat(\bz_k) + \partial r(\bz_{k+1/2}) \ni 0 \\
  \iff & G(\bz_{k+1/2}) + \partial r(\bz_{k+1/2}) \ni \frac{1}{\eta_k}(\bz_k-\bz_{k+1/2}) + G(\bz_{k+1/2}) - \ghat(\bz_k) \\
 \iff & G(\bz_{k+1/2}) + \partial r(\bz_{k+1/2}) \ni \frac{1}{\eta_k}(\bz_k - \bz_{k+1}) + G(\bz_{k+1/2}) - \ghat(\bz_{k+1/2}),
\end{align*}
where the last line used the definition of $\bz_{k+1}$.

Then, in view of the assumption \eqref{eq: weakminty}, we have
\begin{align}
&\langle \eta_k^{-1}(\bz_k - \bz_{k+1}) + G(\bz_{k+1/2}) - \ghat(\bz_{k+1/2}), \bz_{k+1/2} - \bz^\star \rangle \notag \\
&\geq -\rho \| \eta_k^{-1}(\bz_k - \bz_{k+1}) + G(\bz_{k+1/2}) - \ghat(\bz_{k+1/2})\|^2,\label{eq: jdp4}
\end{align}
because $(\bz_{k+1/2}, \eta_k^{-1}(\bz_k - \bz_{k+1}) + G(\bz_{k+1/2}) - \ghat(\bz_{k+1/2})) \in \gra(G+\partial r)$.

We now continue with the standard analysis of FBF-type methods \citep{tseng2000modified}.
Let us first note
\begin{align*}
&2\langle \bz_{k} - \bz_{k+1}, \bz_{k+1/2} - \bz^\star \rangle \\
&= 2\langle \bz_{k} - \bz_{k+1}, \bz_{k+1} - \bz^\star \rangle + 2\langle \bz_{k} - \bz_{k+1},\bz_{k+1/2} - \bz_{k+1} \rangle \\
&= \| \bz_k-\bz^\star\|^2 - \| \bz_{k+1} - \bz^\star \|^2 + \| \bz_{k+1}- \bz_{k+1/2} \|^2 - \| \bz_k-\bz_{k+1/2}\|^2.
\end{align*}
Moreover, by $\bz_{k+1/2}$ being deterministic when we condition on the history up to and including $\bz_{k+1/2}$ and also the unbiasedness of $\ghat$, we have
\begin{align*}
&\E_{k+1/2} \langle G(\bz_{k+1/2}) - \ghat(\bz_{k+1/2}), \bz_{k+1/2} - \bz^\star \rangle\\
&=  \langle \E_{k+1/2}[G(\bz_{k+1/2}) - \ghat(\bz_{k+1/2})], \bz_{k+1/2} - \bz^\star \rangle \\
&= 0.
\end{align*}
We now multiply both sides of \eqref{eq: jdp4} by $2\eta_k$, take expectation, use tower property, and then use the last two estimates in \eqref{eq: jdp4} to get
\begin{align}
\E \| \bz_{k+1} - \bz^\star\|^2 &\leq \E \| \bz_{k} - \bz^\star\|^2 + \E\| \bz_{k+1} - \bz_{k+1/2}\|^2 - \E\| \bz_k-\bz_{k+1/2}\|^2 \notag \\
&\quad + \frac{2\rho}{\eta_k} \E\| \bz_k - \bz_{k+1} + \eta_k(G(\bz_{k+1/2}) - \ghat(\bz_{k+1/2})) \|^2. \label{eq: sim4}
\end{align}
Young's inequality gives
\begin{align*}
 &\| \bz_k - \bz_{k+1} + \eta_k(G(\bz_{k+1/2}) - \ghat(\bz_{k+1/2})) \|^2 \\
 &\leq C_1\| \bz_k - \bz_{k+1/2}\|^2 + C_2\| \bz_{k+1/2} - \bz_{k+1}\|^2 + C_3\eta_k^2\| \ghat(\bz_{k+1/2}) - G(\bz_{k+1/2})\|^2.
\end{align*}

where $C_1 = (1+c_1), \, C_2 = (1+\frac{1}{c_1})(1+c_2),\, C_3 = (1+\frac{1}{c_1})(1+\frac{1}{c_2})$ for any positive $c_1, c_2$.

We next use Young's inequality and Lipschitzness of $G$ to obtain
\begin{align*}
&\|\bz_{k+1} - \bz_{k+1/2}\|^2 \\
&= \eta_k^2 \|\ghat(\bz_k) - \ghat(\bz_{k+1/2})\|^2\\
&\leq \eta_k^2(D_1L^2 \|\bz_k - \bz_{k+1/2}\|^2+  D_2\|\ghat(\bz_k)-G(\bz_k)\|^2 +  D_3\|\ghat(\bz_{k+1/2})-G(\bz_{k+1/2})\|^2).
\end{align*}
where $D_1 = (1+d_1), \, D_2 = (1+\frac{1}{d_1})(1+d_2),\, D_3 = (1+\frac{1}{d_1})(1+\frac{1}{d_2})$ for any positive $d_1, d_2$.

Let us pick
\begin{align*}
c_1 = \frac{1}{\sqrt{6}}, ~~~ c_2 = \frac{1}{30}, ~~~ d_1 = \frac{1}{29} ~~~ d_2 = 1.
\end{align*}
Using these values, plugging in the last two estimates into \eqref{eq: sim4} and combining the like terms conclude the proof.
\end{proof}
Let us restate \Cref{th: mb} and the provide its proof.
\begin{theorem}(Detailed restatement of \Cref{th: mb})\label{th: mb_supp}
Let Assumptions \ref{asp: 1} and \ref{asp: 3} hold and suppose that 
\begin{align*}
\rho < \frac{72}{(360+205\sqrt{6})L} \approx \frac{1}{12L}, ~~~ \eta_k = \frac{1}{\sqrt{6}L}, ~~~ \delta = 1-(2+\frac{2}{\sqrt{6}})\frac{\rho}{\eta_k} - \frac{30}{29}\eta_kL^2(\eta_k+\frac{31(1+\sqrt{6})}{15}\rho)
\end{align*}
where $\delta>0$ by definition, and
\begin{align*}
b_k = \bar{b} (k+1)\log^2(k+3), \text{~where~} \bar{b} = \frac{(306+31\sqrt{6})B^2}{2L^2\delta}.
\end{align*}
Then, the algorithm in \eqref{eq: fbf_stoc} with gradient estimators computed as \eqref{eq: fbf_mb} and parameters as above outputs $\bz^\out$ such that
\begin{align*}
\E [\res(\bz^{\out})] \leq \varepsilon \text{~with stochastic oracle complexity~} \widetilde{O}(\varepsilon^{-4}),
\end{align*}
where $\bz^{\out} = \bz_{\hat k + 1/2}$ and $\hat k$ is selected uniformly at random from $\{ 0, \dots, K-1\}$.
\end{theorem}
\begin{remark}
The range of $\rho$ obtained in this result is indeed rather pessimistic. We did not aim to optimize this constant and it can be improved by using, for example, the two step variant from \cite{pethick2022escaping}. That is, one can analyze the algorithm
\begin{align*}
\text{Define~} \widehat{F}(\bz_k) &:= \bz_k - \eta \ghat(\bz_k), \widehat{F}(\bz_{k+1/2}) := \bz_{k+1/2} - \eta \ghat(\bz_{k+1/2}) \\
\bz_{k+1/2} &= J_{\eta \partial r}(\widehat{F}(\bz_k)) \\
\bz_{k+1/2} &= \bz_k - \tau(\widehat{F}(\bz_k) - \widehat{F}(\bz_{k+1/2})),
\end{align*}
where $\tau$ is smaller than $\eta$ proportionally to obtain a better range for $\rho$, by using the insights from \cite{pethick2023solving,pethick2022escaping}.
In either case, the resulting range for $\rho$ will be smaller than the best-known range given in \Cref{sec: mlmc}. As a result, we skip this variant for brevity.
\end{remark}
\begin{proof}
Let us pick
\begin{align*}
\eta_k^2 = \frac{1}{6L^2},
\end{align*}
where we assume that $\rho < \frac{72}{(360+205\sqrt{6})L} \approx \frac{1}{12L}$. This implies that $\delta = 1-(2+\frac{2}{\sqrt{6}})\frac{\rho}{\eta_k} - \frac{30}{29}\eta_kL^2(\eta_k+\frac{31(1+\sqrt{6})}{15}\rho) > 0$ and then the inequality in Lemma~\ref{lem: app1} becomes
\begin{align}
\delta\E \| \bz_k - \bz_{k+1/2} \|^2 &\leq \E \| \bz^\star- \bz_k\|^2 - \E \| \bz^\star - \bz_{k+1}\|^2 \notag \\
&\quad + \left( \frac{10}{L^2} + \frac{31(6+\sqrt{6})\rho}{L} \right) \E \| \ghat(\bz_{k+1/2}) - G(\bz_{k+1/2})\|^2 \notag \\
&\quad + \left( \frac{10}{L^2} + \frac{62(6+\sqrt{6})\rho}{3L} \right) \E \| \ghat(\bz_{k}) - G(\bz_{k})\|^2.\label{eq: hjk4}
\end{align}
We next bound the terms in the last two lines. By using the fact that $\ghat$ is a mini-batch estimator that averages unbiased i.i.d. samples, we have, by standard estimations (see e.g. \citep[Eq. (3.45)]{kotsalis2022simpleii}) that
\begin{align*}
\E\| \ghat(\bz_{k+1/2}) - G(\bz_{k+1/2}) \|^2 &\leq \frac{1}{b_k} \E \| \gtil(\bz_{k+1/2}, \xi_{k+1/2}^{1}) - G(\bz_{k+1/2})\|^2  \\
&\leq \frac{1}{b_k}\left( B^2\E \| \bz_{k+1/2} - \bz_0 \|^2 + \sigma^2  \right) \\
&\leq \frac{3B^2}{b_k} \E \left[ \| \bz_{k+1/2} - \bz_k\|^2 + \| \bz_k - \bz^\star \|^2 + \| \bz^\star - \bz_0\|^2 \right] + \frac{\sigma^2}{b_k}
\end{align*}
and similarly,
\begin{align*}
\E\| \ghat(\bz_{k}) - G(\bz_{k}) \|^2 &\leq \frac{2B^2}{b_k} \E \left[ \| \bz_k - \bz^\star \|^2 + \| \bz^\star - \bz_0\|^2  \right] + \frac{\sigma^2}{b_k}.
\end{align*}
After plugging in the last two estimates into \eqref{eq: hjk4} and using 
\begin{equation*}
\bar{b} = \frac{(306+31\sqrt{6})B^2}{2L^2\delta} \text{~and~} b_k = \bar{b} (k+1)\log^2(k+3),
\end{equation*}

we obtain
\begin{align}
\frac{\delta}{2}\E \| \bz_k - \bz_{k+1/2} \|^2 &\leq \left( 1+\frac{\delta}{(k+1)\log^2(k+3)} \right) \E \| \bz^\star- \bz_k\|^2 - \E \| \bz^\star - \bz_{k+1}\|^2 \notag \\
&\quad + \frac{1}{(k+1)\log^2(k+3)}\left( \delta \| \bz^\star- \bz_0\|^2+ \frac{60\sigma^2}{L^2\bar{b}} \right).\label{eq: hjk5}
\end{align}
First, we discard the nonnegative term on the left-hand side and obtain
\begin{align}
\E \| \bz^\star - \bz_{k+1}\|^2 &\leq \left( 1+\frac{\delta}{(k+1)\log^2(k+3)} \right) \E \| \bz^\star- \bz_k\|^2  \notag \\
&\quad + \frac{1}{(k+1)\log^2(k+3)}\left( \delta \| \bz^\star- \bz_0\|^2+ \frac{60\sigma^2}{L^2\bar{b}} \right).\label{eq: df44}
\end{align}
Note how this recursion is of the same form as \citep[Lemma 5.31]{bauschke2017convex} because $\sum_{k=0}^\infty\frac{1}{(k+1)\log^2(k+3)} < 2$, hence arguing in the same way gives us that 
\begin{equation}\label{eq: zk_unif_bd}
\E \| \bz^\star - \bz_k\|^2\leq C < +\infty.
\end{equation}
for an easily computable constant $C$ that depends on the constants in \eqref{eq: df44}. Then plugging in this uniform upper bound on $\E \| \bz^\star - \bz_k\|^2$ into \eqref{eq: hjk5} gives
\begin{align}\label{eq: zk_zkhalf}
\frac{1}{K}\sum_{k=0}^{K-1} \E \| \bz_k - \bz_{k+1/2} \|^2 = O\left( \frac{1}{K}\right),
\end{align} 
since $\sum_{k=0}^\infty \frac{1}{(k+1)\log^2(k+3)} < 2$.

We finally bound the quantity $\frac{1}{K}\sum_{k=1}^K \E[\res^2(\bz_{k+1/2})]=\frac{1}{K}\sum_{k=1}^K \E[\min_{\bu \in (G+\partial r)\bz_{k+1/2}} \|\bu\|^2]$. For this,  note that
\begin{align*}
&\bz_{k+1/2} = \arg\min_{\bz} r(\bz) + \frac{1}{2\eta_k} \| \bz-(\bz_k-\eta_k \ghat(\bz_k)) \|^2 \\
 \iff &\bz_{k+1/2} + \eta_k \partial r(\bz_{k+1/2}) \ni \bz_k - \eta_k \ghat(\bz_k) \\
 \iff & G(\bz_{k+1/2}) + \partial r(\bz_{k+1/2}) \ni \eta_k^{-1}(\bz_k - \bz_{k+1/2}) + G(\bz_{k+1/2}) - \ghat(\bz_{k}),
\end{align*}
which implies by Young's inequalities that
\begin{align*}
\E[\res^2(\bz_{k+1/2})] &= \E [\min_{\bu \in (G+\partial r)\bz_{k+1/2}}  \| \bu\|^2] \leq 3(\eta_k^{-2}+L^2)\| \bz_{k}-\bz_{k+1/2}\|^2 + 3\|G(\bz_k)-\ghat(\bz_k)\|^2 \\
&\leq 3(\eta_k^{-2}+L^2)\| \bz_{k}-\bz_{k+1/2}\|^2 + \frac{3}{\bar{b}(k+1)\log^2(k+3)}\left(B^2\|\bz_k-\bz_0\|^2 + \sigma^2\right),
\end{align*}
where we used the property of the mini-batch estimator $\widehat{G}(\bz_k)$.

This gives that
\begin{align*}
&\frac{1}{K}\sum_{k=0}^{K-1}\E[\res^2(\bz_{k+1/2})] \\
& \leq \frac{1}{K}\sum_{k=0}^{K-1}\left( 3(\eta_k^{-2}+L^2)\| \bz_{k}-\bz_{k+1/2}\|^2 + \frac{3}{\bar{b}(k+1)\log^2(k+3)}\left(B^2\|\bz_k-\bz_0\|^2 + \sigma^2\right) \right).
\end{align*}
We now plug in \eqref{eq: zk_unif_bd} and \eqref{eq: zk_zkhalf} here (after applying Young's inequality) and use $\sum_{k=0}^\infty \frac{1}{(k+1)\log^2(k+3)} < +\infty$ to obtain 
\begin{align*}
\frac{1}{K}\sum_{k=0}^{K-1}\E[\res^2(\bz_{k+1/2})] = O\left( \frac{1}{K} \right).
\end{align*}
Hence, the number of iterations $K$ to make the left-hand side less than $\varepsilon$ is of the order $\varepsilon^{-2}$.
Moreover, the mini-batch sizes over $\{0, \dots, K-1\}$ are upper bounded by $b_k\leq \bar{b}K\log(K+2)$, giving the stochastic oracle complexity $\widetilde{O}(\varepsilon^{-4})$ for getting
\begin{equation*}
\E [\res^2(\bz_{\hat k + 1/2})] \leq\varepsilon^2,
\end{equation*}
where $\hat k$ is selected uniformly at random from $\{ 0,\dots, K-1\}$. Note also that we use Jensen's inequality to conclude the assertion in the theorem statement since $\E [\res(\bz_{\hat k + 1/2})] \leq \sqrt{\E[\res^2(\bz_{\hat k + 1/2})]}$ where $\bz^{\out} = \bz_{\hat k + 1/2}$.
\end{proof}
\newpage 

\section{Proofs for Section \ref{sec: mlmc}}\label{app: proofs3}
In this case, the main departure from the work of \cite{alacaoglu2024revisiting} is the realization that a bounded variance assumption is not required for their construction to go through. That is, the only place that bounded variance is needed in this paper is for the inner solver used in the MLMC estimator. This inner solver is an operator splitting algorithm applied to a strongly monotone problem.
Thanks to strong monotonicity, it is rather straightforward to handle \Cref{asp: 3} which replaces the bounded variance assumption, with a small change on the parameter choices.
Then we will see that we can use this analysis for the inner solver under \Cref{asp: 3} and obtain the same complexity guarantees as \cite{alacaoglu2024revisiting}, which also allows us to handle problems where $\rho$ has the best-known upper bound.

We start with the analysis of the inner solver. It is worth noting that different operator splitting methods are already analyzed under \Cref{asp: 3} and strong monotonicity, see for example \cite{iusem2017extragradient} and \cite{kotsalis2022simpleii}. We provide a proof for FBF under this case for being self-contained and then we show how to use this result in the construction of \cite{alacaoglu2024revisiting} for getting the final result.

In particular, in the innermost loop of our algorithm, we are solving a strongly monotone inclusion problem. Hence let us write down the abstract problem (and we will see later how to map this back to our original setting)
\begin{equation}\label{eq: aas4}
0\in (A+\mathsf{B})\bx^\star,
\end{equation}
where we have access to $\widetilde{B}$ such that $\E [\widetilde{B}(\bx)] = \mathsf{B}(\bx)$. For terms such as strong monotonicity or maximal monotonicity, we refer to the textbook \cite{bauschke2017convex}. What matters for our purposes is that our subproblem, that is, estimation of $J_{\eta(G+\partial r)}(\bz_k)$ satisfies these assumptions.
\begin{lemma}(Detailed restatement of \Cref{eq: siv4})\label{lem: str_mon}
Let $A$ in \eqref{eq: aas4} be maximally monotone and $\bsf$ be $L_\bsf$-Lipschitz and $\mu$-strongly monotone. Assume that $\E \| \btil(\bx) - \bsf(\bx)\|^2 \leq B^2 \| \bx-\bx_0\|^2 + \sigma^2$.
Let $\tau_t = \frac{4}{(t+1)\mu+144M^2/\mu}$. Then, we have for the output of \Cref{alg:fbf_stoc} that
\begin{equation*}
\E\|\bx^\star-\bx_{T}\|^2\le \frac{3\kappa\|\bx^\star-\bx_0\|^2+282\sigma^2/\mu^2}{T+\kappa},
\end{equation*}
where $\kappa = 144M^2/\mu^2$ with $M=\max(L_\bsf, B)$.
\end{lemma}
\begin{remark}
Note that the above lemma is written for an arbitrary problem of finding $\bx^\star$ such that $0\in (A+\bsf)\bx^\star$. However, in our particular case, our subproblem (that is, finding the resolvent $J_{\eta(G+\partial r)}(\bz_k)$) is finding $\bx$ such that $0\in (\id+\eta G)\bx + \eta \partial r(\bx) - \bz_k$, hence in the notation of the above statement, we have $\bx^\star = J_{\eta(G+\partial r)}(\bz_k)$ and then the left-hand side of the statement becomes $\E\| \bx_T - J_{\eta(G+\partial r)}(\bz_k)\|^2$. The operators are mapped as $A = \eta \partial r$ and $\bsf(\cdot) = (\id+\eta G)(\cdot) - \bz_k$.
\end{remark}
\begin{proof}
For brevity, let us denote $L=L_\bsf$ in this proof. We proceed as \citep[Theorem C.1, until Equation (50)]{alacaoglu2024revisiting}. By using strong monotonicity of $G$ and monotonicity of $\partial r$, one obtains
\begin{align}
    \left( \frac{1}{2\tau_t} + \frac{\mu}{2} \right) \E \| \bx^\star - \bx_{k+1} \|^2 &\leq \frac{1}{2\tau_t} \E \| \bx^\star - \bx_k \|^2 + \frac{19}{36\tau_t} \E \| \bx_{t+1} - \bx_{t+1/2}\|^2 \notag \\
     &\quad-\frac{1}{2\tau_t} \E\|\bx_t - \bx_{t+1/2}\|^2,\label{eq: scx4}
\end{align}
since $\eta_t \mu < 1/36$.

We estimate using Young's inequalities, Lipschitzness of $\bsf$, and \Cref{asp: 3} to obtain
\begin{align*}
    &\E\| \bx_{t+1} - \bx_{t+1/2}\|^2 \leq \tau_t^2 \| \btil(\bx_t) - \btil(\bx_{t+1/2})\|^2 \\
    &\leq 3\tau_t^2 \E\left( \| \btil(\bx_t) - \bsf(\bx_t)\|^2 + \| \btil(\bx_{t+1/2}) - \bsf(\bx_{t+1/2})\|^2 + \| \bsf(\bx_t) - \bsf(\bx_{t+1/2})\|^2 \right) \\
    &\leq 3\tau_t^2  B^2 \E\| \bx_t - \bx_0\|^2 + 3\tau_t^2  B^2 \E\| \bx_{t+1/2} - \bx_0\|^2 + 6\tau_t^2 \sigma^2 + 3\tau_t^2 L^2 \E\| \bx_t-\bx_{t+1/2}\|^2\\
    &\leq 6\tau_t^2 B^2 \left(\E\| \bx_t- \bx^\star\|^2 + \| \bx_0-\bx^\star\|^2 \right) \\
    &\quad+ 9\tau_t^2 B^2 \left( \E\| \bx_{t} - \bx^\star\|^2 + \E\| \bx_t - \bx_{t+1/2}\|^2 + \| \bx_0-\bx^\star\|^2 \right) \\
    &\quad + 3\tau_t^2L^2\E\|\bx_t-\bx_{t+1/2}\|^2+ 6\tau_t^2\sigma^2.
\end{align*}
We plug this into \eqref{eq: scx4}, use $\frac{19}{36\tau_t} \times 15\tau_t^2B^2 < 8\tau_t B^2$, and the notation $M = \max(L, B)$ to get
\begin{align*}
    \left( \frac{1}{2\tau_t} + \frac{\mu}{2} \right) \E\|\bx^\star - \bx_{t+1}\|^2 &\leq \left( \frac{1}{2\tau_t} + 8\tau_t B^2 \right) \E\| \bx^\star - \bx_t\|^2 + 8\tau_tB^2 \| \bx_0-\bx^\star\|^2 \\
    &\quad+ \left( 10\tau_tM^2 - \frac{1}{2\tau_t} \right)\E\| \bx_t-\bx_{t+1/2}\|^2 + 5\tau_t \sigma^2.
\end{align*}
To simplify, let us assume that $8\tau_t B^2 \leq \frac{\mu}{4}$ and define $\gamma_t$ such that
\begin{equation}\label{eq: shd5}
\frac{1}{2\gamma_t} = \frac{1}{2\tau_t} + \frac{\mu}{4}.
\end{equation}
Then, we can equivalently write the previous recursion as
\begin{align}
    \left( \frac{1}{2\gamma_t} + \frac{\mu}{4} \right) \|\bx^\star - \bx_{t+1}\|^2 &\leq  \frac{1}{2\gamma_t} \| \bx^\star - \bx_t\|^2 + 8\tau_tB^2 \| \bx_0-\bx^\star\|^2 \notag \\
    &\quad+ \left( 10\tau_tM^2 - \frac{1}{2\eta_t} \right)\| \bx_t-\bx_{t+1/2}\|^2 + 5\tau_t \sigma^2.\label{eq: asx4}
\end{align}
Hence, the main recursion we have is very similar to \citep[Eq. (51)]{alacaoglu2024revisiting}.

Assume now
\begin{equation*}
10\tau_tM^2 \leq \frac{1}{2\tau_t} \iff 20\tau_t^2 M^2 \leq 1.
\end{equation*}
Let us now set 
\begin{equation*}
\gamma_t = \frac{4}{(t+3)\mu+144M^2/\mu}, \text{~} \frac{1}{2\gamma_t} = \frac{(t+3)\mu+144M^2/\mu}{8} \text{~and~} \frac{1}{2\gamma_t} + \frac{\mu}{4} = \frac{(t+5)\mu+144M^2/\mu}{8}.
\end{equation*}
Let us note that this also gives the following for $\tau_t$ (see \eqref{eq: shd5}):\begin{equation*}
\frac{1}{2\tau_t} = \frac{(t+1)\mu+144M^2/\mu}{8} \iff \tau_t = \frac{4}{(t+1)\mu+144M^2/\mu}.
\end{equation*}
Notice how both of the following requirements are satisfied with this choice of $\tau_t$ (by recalling that $M=\max(L, B)$):
\begin{equation*}
20\tau_t^2M^2 \leq 1\text{~~~and~~~} 8\tau_tB^2 \leq \frac{\mu}{4}.
\end{equation*}
Next, multiply \eqref{eq: asx4} by $\left(\frac{1}{2\gamma_t} + \frac{\mu}{4}\right)^{-1} = \frac{8}{(t+5)\mu+144M^2/\mu}$ and get
\begin{align*}
\|\bx^\star-\bx_{t+1}\|^2 &\leq \frac{(t+3)\mu+144M^2/\mu}{(t+5)\mu+144M^2/\mu} \| \bx^\star-\bx_t\|^2 + \frac{64\tau_tM^2}{(t+5)\mu+144M^2/\mu} \| \bx_0-\bx^\star\|^2 \\
&\quad+ \frac{64}{(t+5)\mu+144M^2/\mu}\tau_t\sigma^2.
\end{align*}
Plugging in $\tau_t$ gives
\begin{align*}
\|\bx^\star-\bx_{t+1}\|^2 &\leq \frac{(t+3)\mu+144M^2/\mu}{(t+5)\mu+144M^2/\mu} \| \bx^\star-\bx_t\|^2 \\
&\quad + \frac{256M^2}{((t+5)\mu+144M^2/\mu)((t+1)\mu+144M^2/\mu)} \| \bx_0-\bx^\star\|^2 \\
&\quad + \frac{256}{((t+5)\mu+144M^2/\mu)((t+1)\mu+144M^2/\mu)}\sigma^2.
\end{align*}
An equivalent way to write this, by letting $\kappa=144M^2/\mu^2$ is
\begin{align*}
\|\bx^\star-\bx_{t+1}\|^2 &\leq \frac{t+3+\kappa}{t+5+\kappa} \| \bx^\star-\bx_t\|^2 \\
&\quad + \frac{2\kappa}{(t+5+\kappa)(t+1+\kappa)} \| \bx_0-\bx^\star\|^2 \\
&\quad + \frac{256\sigma^2/\mu^2}{(t+5+\kappa)(t+1+\kappa)}.
\end{align*}
We now prove by induction that
\begin{equation*}
\E\|\bx^\star-\bx_{t}\|^2\le \frac{a\kappa\|\bx^\star-\bx_0\|^2+bG^2/\mu^2}{t+\kappa},
\end{equation*}
with $a=3$ and $b=282$.

The base case $t=0$ holds trivially when $a\geq 1$.
We want to show that
\begin{align*}
\| \bx^\star-\bx_{T+1}\|^2 &\leq \frac{T+3+\kappa}{T+5+\kappa} \frac{a\kappa\|\bx^\star-\bx_0\|^2 + b\sigma^2/\mu^2}{T+\kappa}  + \frac{2\kappa}{(T+5+\kappa)(T+1+\kappa)} \| \bx^\star-\bx_0\|^2 \\
&\quad + \frac{256\sigma^2/\mu^2}{(T+5+\kappa)(T+1+\kappa)}.
\end{align*}
For some constant $d > 1$, let us have $\frac{2}{a}\leq \frac{1}{d}$ and $\frac{256}{b}\leq \frac1d$, and then the previous bound implies
\begin{align*}
\| \bx^\star-\bx_{T+1}\|^2 \leq \frac{1}{T+5+\kappa} \left(\frac{T+3+\kappa}{T+\kappa} + \frac{1}{d(T+1+\kappa)}\right) \left(a\kappa\|\bx^\star-\bx_0\|^2 + b\sigma^2/\mu^2 \right)
\end{align*}
and then we wish to find $d$ such that (after setting $A = T+\kappa=T+144M^2/\mu^2$ and where $A \geq 144$)
\begin{align*}
&\frac{A+3}{A(A+5)} + \frac{1}{d(A+1)(A+5)} \leq \frac{1}{A+1} \iff d(A+1)(A+3) + A \leq dA(A+5) \\
&\iff 4A d + 3d + A \leq 5A d \iff  3d+A \leq dA \iff 3d \leq (d-1)A,
\end{align*}
which is satisfied when $d=1.1$ since this would require $3.3 \leq 0.1A \iff 33\leq A$ which is true because $A =T+\kappa \geq \kappa \geq 144$. We then use the bounds for $a, b$ as
\begin{equation*}
a \geq 2.2\text{~~~and~~~} b \geq 256\times 1.1,
\end{equation*}
which are satisfied with $a=3, b=282$.

The proof is completed.
\end{proof}

\subsection{Results for the MLMC estimator}
We now continue with the results related to the MLMC estimator, taken from \cite{alacaoglu2024revisiting}. The only difference will be the change in the complexity analysis inner solver used for estimating the resolvent (that is, the proximal subproblem). As a result, the only change in the lemmas we cite below are the constants.

For brevity, let us denote the upper bound in Lemma~\ref{lem: str_mon} as
\begin{align}\label{eq: rate_bd}
C = 3\kappa\|\bx^\star-\bx_0\|^2+282\sigma^2/\mu^2, ~~~ C_1 = 3\kappa, ~~~ C_2 = 282/\mu^2.
\end{align}
where $\kappa = 144M^2/\mu^2$.
The lemmas below have identical proofs to the corresponding results we cite from \cite{alacaoglu2024revisiting} with the minor differences of having the constants in Lemma~\ref{lem: str_mon} instead of \citep[Theorem C.1]{alacaoglu2024revisiting}. Hence, we do not repeat their proofs and just refer to \cite{alacaoglu2024revisiting}.

\begin{lemma}(\citep[Lemma C.9]{alacaoglu2024revisiting}, \citep[Property 1]{asi2021stochastic})\label{lem: mlmc1}
Under the setting of Lemma~\ref{lem: str_mon}, we have, for the output of \Cref{alg:fbf_mlmc}, that
\begin{align*}
\| \E[\by^{\out}] - \by^\star \|^2 &\leq \frac{2C}{N},\\
\E \| \by^\out - \by^\star \|^2 &\leq 14C\log_2N.
\end{align*}
where $C$ is as defined in \eqref{eq: rate_bd} and the number of calls to $\btil$ is $O(\log_2N)$.
\end{lemma}

\begin{lemma}(\citep[Corollary C.10]{alacaoglu2024revisiting}, \citep[Theorem 1]{asi2021stochastic})\label{lem: mlmc2}
Let $\widetilde{J}_{\eta(G+\partial r)}(\bz_k)$ be as defined in \Cref{alg:weakmvi_stoc}. Under the setting of Lemma~\ref{lem: str_mon}, we have for any $b_k, v$ that
\begin{align*}
\| \E[\widetilde{J}_{\eta(G+\partial r)}(\bz_k)] - J_{\eta(G+\partial r)}(\bz_k) \|^2 &\leq b_k^2(\|(\id-J_{\eta(G+\partial r)})(\bz_k)\|^2 + \sigma^2),\\
\E \| \widetilde{J}_{\eta(G+\partial r)}(\bz_k) - J_{\eta(G+\partial r)}(\bz_k) \|^2 &\leq v^2(\|(\id-J_{\eta(G+\partial r)})(\bz_k)\|^2 + \sigma^2),
\end{align*}
with the parameters of Alg. \ref{alg:weakmvi_stoc} selected as
\begin{align}\label{eq: sar4}
N_k = \left\lceil \frac{\max\{ 2C_1, 2C_2\}}{\min\{b_k^2, v^2/2\}} \right\rceil, \text{~and~} M_k =  \left\lceil \frac{28\max\{ C_1, C_2\} \log_2 N_k }{v^2} \right\rceil
\end{align}
where $C_1, C_2$ are as defined in \eqref{eq: rate_bd}. The number of calls to the stochastic first-order oracle at each iteration is $O((\log N_k)\cdot M_k)$ in expectation.
\end{lemma}
Finally, the complexity analysis will be the combination of these three results. Most of the derivation is the same as \cite{alacaoglu2024revisiting} up to the change of the inner solver, as a result, we only sketch the differences in the proof, compared to \citep[Theorem C.11]{alacaoglu2024revisiting}.

\begin{theorem}(Detailed restatement of \Cref{th: mlmc})\label{th: mlmc_supp}
Let Assumptions \ref{asp: 1} and \ref{asp: 3} hold and suppose that $\rho < 1/L$. Then, for \Cref{alg:weakmvi_stoc} with $\eta \leq \frac{1}{L}$ and $\alpha=1-\rho/\eta$, $\alpha_k = \frac{\alpha}{\sqrt{k+2}\log(k+3)}$; and $N_k, M_k$ given in \eqref{eq: sar4} with $b_k^2 = \frac{\alpha_k}{120\alpha(k+1)}, v^2=\frac{1}{60}$,
we can generate $\bz^\out$ such that
\begin{equation*}
\E [\res(\bz^{\out})] \leq \varepsilon \text{~with expected stochastic oracle complexity~} \widetilde{O}(\varepsilon^{-4}),
\end{equation*}
where $\bz_{\hat k}$ is selected uniformly at random after running the algorithm for $K$ iterations and $\bz^{\out}$ is generated by applying one step of \eqref{eq: fbf_stoc} for problem \eqref{eq: def_prox}, starting from $\bz_{\hat k}$.
\end{theorem}
\begin{proof}
The proof of this theorem mirrors that of \citep[Theorem C.11]{alacaoglu2024revisiting} which is in fact oblivious to the inner solver until the last paragraph of the proof, particularly the calculation of the stochastic oracle complexity of each iteration. In our case, $N_k, M_k$ depend on $C_1, C_2$ which are slightly different than \cite{alacaoglu2024revisiting}, due to us using Lemma~\ref{lem: str_mon} which requires only \Cref{asp: 3}, unlike the result of \cite{alacaoglu2024revisiting} that used the bounded variance assumption. However, this only changes the absolute constants for the number of expected calls to the stochastic oracle and the complexity result is hence the same. That is, we plug Lemma~\ref{lem: str_mon} into \Cref{lem: mlmc1} and \Cref{lem: mlmc2}, then use these bounds for the right-hand side of \Cref{eq: lem_resid} to get the complexity result for the average of $\E \| \id-J_{\eta(G+\partial r)}(\bz_k)\|^2$.

Finally converting a bound on $\E \| \id-J_{\eta(G+\partial r)}(\bz)\|^2$ to a guarantee on $\E[\res(\bz^\out)]$ is standard, see for example the textbook \citep[Lemma 6.3]{lan2020first}. A similar argument by using a mini-batch estimator gives the result.
\end{proof}

\newpage
\section{Proofs for Section \ref{sec: vr}}\label{app: proofs4}
Let us provide a further intuition for \Cref{alg: var_red_fbf}.
Note that by the definition of $\bz_{k+1/2}$, we have
\begin{align*}
\iff &0 \in \partial r(\bz_{k+1/2}) + \frac{1}{\gamma_k}\left( \bz_{k+1/2} - \bar\bz_k + \gamma_k \bg_k \right) \\
\iff &G(\bz_{k+1/2}) + \partial r(\bz_{k+1/2})\ni\frac{1}{\gamma_k}\left( \bar\bz_k - \bz_{k+1/2}-\gamma_k\bg_k + \gamma_k G(\bz_{k+1/2}) \right).
\end{align*}
Hence, the second step, that is, the update of $\bz_{k+1}$ is a subgradient descent-like step by using the unbiased estimate of a particular \emph{subgradient} (indeed for a min-max problem, $G$ contains the gradients of the coupling function) from $(G+\partial r)\bz_{k+1/2}$ with step size $\tau_k$ (which absorbs the $1/\gamma_k$ appearing above). In this update, we move from the anchored iterate $\bar\bz_k$ to follow the idea of Halpern anchoring.

We now analyze one iteration of this method.
\begin{lemma}\label{lem: vr_one_it}
Let Assumptions \ref{asp: 1}, \ref{asp: 2}, and \ref{asp: 3} hold. Then we have, for \Cref{alg: var_red_fbf}, that
\begin{align}
\E\| \bz^\star-\bz_{k+1}\|^2 &\leq (1-\beta_k) \E\| \bz^\star-\bz_k\|^2 + \beta_k \| \bz^\star-\bz_0\|^2 - \frac{\beta_k(1-\tau_k)}{1+\tau_k} \E\|\bz_0-\bz_{k+1}\|^2 \notag \\
&\quad + \mathcal{C}_{1,k} \E\|\bz_k-\bz_{k+1/2}\|^2 + \mathcal{C}_{2,k} \E\| \bz_k-\bz_{k+1}\|^2  \notag \\
&\quad + \mathcal{C}_{3,k} \E\| \bz_0-\bz_{k+1/2}\|^2 + \mathcal{C}_{4,k}\E\|\bg_k - G(\bz_k)\|^2 + \mathcal{E}_{k}.\label{eq: sfx4}
\end{align}
where we define,
\begin{equation}\label{eq: c_defs}
\begin{aligned}
    \mathcal{E}_k &=  \frac{72\tau_k^2\gamma_k^2\sigma^2}{1+\tau_k}, \\
    \mathcal{C}_{1,k} &= \frac{\tau_k(1/4+(9/2)\gamma_k^2L^2 + 72\tau_k\gamma_k^2 L^2 - 2(1-\beta_k))}{1+\tau_k}+\frac{5\rho\tau_k(1-\beta_k)}{2\gamma_k}+{11\rho\tau_kL^2\gamma_k},\\
    \mathcal{C}_{2,k} &= \frac{1}{24(1+\tau_k)} - \frac{(1-\tau_k)(1-\beta_k)}{1+\tau_k}, \\
\mathcal{C}_{3,k} &= \frac{5\rho\tau_k\beta_k}{2\gamma_k} + \frac{2\tau_k(36\tau_k\gamma_k^2 B^2 -\beta_k)}{1+\tau_k},\\
    \mathcal{C}_{4,k} &= 110\rho\tau_k\gamma_k +\frac{4\tau_k\gamma_k^2(9+18\tau_k)}{1+\tau_k}.
\end{aligned}
\end{equation}
\end{lemma}
\begin{remark}
The bound on the right-hand side of \eqref{eq: sfx4} is rather complicated. The main intuition is that for each of the error terms independent of $\bz^\star$, we have negative coefficients that we can use to cancel the positive coefficients (after picking the parameters accordingly), except the last terms on the right-hand side of \eqref{eq: sfx4}. For the second term of the last line, we will use the classical bound of STORM variance reduced estimator of \cite{cutkosky2019momentum}. The last term in the last line will be sufficiently small due to the choice of $\tau_k^2$.
\end{remark}
\begin{proof}[Proof of Lemma~\ref{lem: vr_one_it}]
The definitions of $\bz_{k+1/2}$ and the proximal operator give
\begin{align*}
&\bz_{k+1/2} + \gamma_k \partial r(\bz_{k+1/2}) \ni \bar \bz_k - \gamma_k \bg_k \\
\iff & G(\bz_{k+1/2}) + \partial r(\bz_{k+1/2}) \ni \frac{1}{\gamma_k}\left( \bar\bz_k - \bz_{k+1/2} \right) -\bg_k + G(\bz_{k+1/2}).
\end{align*}
Then, using \eqref{eq: weakminty} gives
\begin{align}
&\left\langle \frac{1}{\gamma_k}(\bar\bz_k - \bz_{k+1/2}) -  \bg_k + G(\bz_{k+1/2}), \bz_{k+1/2} - \bz^\star \right\rangle \notag \\
& \geq -\rho \left\|\frac{1}{\gamma_k}(\bar\bz_k - \bz_{k+1/2}) -  \bg_k + G(\bz_{k+1/2}) \right\|^2.
\end{align}
After multiplying both sides by $\tau_k\gamma_k$, we get
\begin{align}
&\tau_k\langle \bar\bz_k - \bz_{k+1/2} - \gamma_k \bg_k, \bz_{k+1/2} - \bz^\star \rangle + \tau_k\gamma_k \langle G(\bz_{k+1/2}), \bz_{k+1/2} - \bz^\star \rangle \notag \\
&\quad \geq -\frac{\rho \tau_k}{\gamma_k} \| \bar\bz_k - \bz_{k+1/2} - \gamma_k \bg_k + \gamma_k G(\bz_{k+1/2}) \|^2.\label{eq: ake3}
\end{align}
We estimate the first inner product.
The definition of $\bz_{k+1}$, after rearranging and dividing each side by $\tau_k$ yields
\begin{align*}
\bz_{k+1/2} - \bar\bz_k = \frac{1}{\tau_k}\left( \bz_{k+1} - \bar\bz_k \right)- \gamma_k \bg_k +  \gamma_k \gtil(\bz_{k+1/2}, \xi_{k+1/2}).
\end{align*}
As a result, we have for the first inner product in \eqref{eq: ake3} that
\begin{align*}
\tau_k\langle  \bz_{k+1/2} - \bar\bz_k + \gamma_k \bg_k, \bz^\star-\bz_{k+1/2}  \rangle=\langle \bz_{k+1} - \bar\bz_k + \tau_k\gamma_k \widetilde{G}(\bz_{k+1/2}, \xi_{k+1/2}), \bz^\star-\bz_{k+1/2} \rangle.
\end{align*}
This gives in \eqref{eq: ake3} that
\begin{align*}
&\langle \bz_{k+1} - \bar\bz_k,  \bz^\star - \bz_{k+1/2} \rangle + \tau_k\gamma_k \langle G(\bz_{k+1/2}) - \gtil(\bz_{k+1/2}, \xi_{k+1/2}), \bz_{k+1/2} - \bz^\star \rangle \notag \\
& \geq -\frac{\rho \tau_k}{\gamma_k} \| \bar\bz_k - \bz_{k+1/2} - \gamma_k \bg_k + \gamma_k G(\bz_{k+1/2}) \|^2.
\end{align*}
After taking expectation, using the tower property, and the assumption of $\gtil(\bz_{k+1/2}, \xi_{k+1/2})$ being unbiased, we obtain
\begin{align}\label{eq: nbw3}
\E\langle \bz_{k+1} - \bar\bz_k,  \bz^\star - \bz_{k+1/2} \rangle \geq -\frac{\rho \tau_k}{\gamma_k} \E\| \bar\bz_k - \bz_{k+1/2} - \gamma_k \bg_k + \gamma_k G(\bz_{k+1/2}) \|^2.
\end{align}
We rewrite the inner product using squared norms to derive
\begin{align}
2\langle \bz_{k+1} - \bar\bz_k, &\bz^\star-\bz_{k+1/2} \rangle = 2\langle \bz_{k+1} - \bz_{k+1/2}, \bz^\star-\bz_{k+1/2} \rangle + 2\langle \bz_{k+1/2} - \bar\bz_k, \bz^\star-\bz_{k+1/2} \rangle \notag \\
&= \| \bz_{k+1} - \bz_{k+1/2}\|^2 - \| \bz^\star-\bz_{k+1}\|^2 - \| \bz_{k+1/2} - \bar\bz_k\|^2 + \|\bz^\star-\bar\bz_k\|^2.\label{eq: jqp4}
\end{align}
We continue to estimate the main error term $\|\bz_{k+1} - \bz_{k+1/2}\|^2$ as
\begin{align}
&\| \bz_{k+1} - \bz_{k+1/2}\|^2 \notag \\
&= \langle \bz_{k+1} - \bz_{k+1/2}, \bz_{k+1} - \bz_{k+1/2} \rangle \notag \\
&= (1-\tau_k)\langle \bar\bz_k - \bz_{k+1/2}, \bz_{k+1} - \bz_{k+1/2} \rangle + \tau_k\gamma_k\langle \bg_k - \gtil(\bz_{k+1/2}, \xi_{k+1/2}), \bz_{k+1} - \bz_{k+1/2} \rangle \notag \\
&= \frac{1-\tau_k}{2} \left( \| \bar\bz_k - \bz_{k+1/2} \|^2 + \| \bz_{k+1} - \bz_{k+1/2}\|^2 - \| \bz_{k+1} - \bar\bz_k\|^2 \right) \notag \\
&\quad+ \tau_k\gamma_k\langle \bg_k - \gtil(\bz_{k+1/2}, \xi_{k+1/2}), \bz_{k+1} - \bz_{k+1/2} \rangle,\label{eq: cst2}
\end{align}
where the second equality used the definition of $\bz_{k+1}$ as
\begin{align*}
\bz_{k+1} - \bz_{k+1/2} &= \bar\bz_k - \bz_{k+1/2} - \tau_k(\bar\bz_k - \bz_{k+1/2} - \gamma_k \bg_k + \gamma_k \gtil(\bz_{k+1/2}, \xi_{k+1/2}))\\
&=(1-\tau_k)(\bar\bz_k - \bz_{k+1/2}) + \tau_k\gamma_k(  \bg_k - \gtil(\bz_{k+1/2}, \xi_{k+1/2})).
\end{align*}
The last equality in \eqref{eq: cst2} used the elementary identity $\langle \ba, \bb \rangle = \frac{1}{2} \left( \| \ba\|^2 + \|\bb\|^2 - \| \ba-\bb\|^2 \right)$.

Rearranging \eqref{eq: cst2} leads to
\begin{align*}
&\frac{1+\tau_k}{2} \| \bz_{k+1} - \bz_{k+1/2}\|^2 \\
&= \frac{1-\tau_k}{2} \left(\| \bar\bz_k - \bz_{k+1/2} \|^2 - \| \bz_{k+1} - \bar\bz_k\|^2 \right)+ \tau_k\gamma_k\langle \bg_k - \gtil(\bz_{k+1/2}, \xi_{k+1/2}), \bz_{k+1} - \bz_{k+1/2} \rangle.
\end{align*}
Multiplying both sides of $\frac{2}{1+\tau_k}$ gives us
\begin{align*}
\| \bz_{k+1} - \bz_{k+1/2}\|^2 &= \frac{1-\tau_k}{1+\tau_k} \left( \| \bar\bz_k - \bz_{k+1/2} \|^2 - \| \bz_{k+1} - \bar\bz_k\|^2 \right) \\
&\quad+ \frac{2\tau_k\gamma_k}{1+\tau_k} \langle \bg_k - \gtil(\bz_{k+1/2}, \xi_{k+1/2}), \bz_{k+1} - \bz_{k+1/2} \rangle.
\end{align*}
With this, we estimate \eqref{eq: jqp4} as
\begin{align}
&2\langle \bz_{k+1} - \bar\bz_k, \bz^\star-\bz_{k+1/2} \rangle \notag \\
&= \frac{-2\tau_k}{1+\tau_k} \| \bar\bz_k - \bz_{k+1/2}\|^2 - \frac{1-\tau_k}{1+\tau_k} \| \bz_{k+1} - \bar\bz_k\|^2 - \| \bz^\star-\bz_{k+1}\|^2+ \|\bz^\star-\bar\bz_k\|^2 \notag \\
&\quad+ \frac{2\tau_k\gamma_k}{1+\tau_k} \langle \bg_k - \gtil(\bz_{k+1/2}, \xi_{k+1/2}), \bz_{k+1} - \bz_{k+1/2} \rangle.\label{eq: ewr3}
\end{align}
Next, we use the definition of $\bar\bz_k$ for estimating the first, second and fourth terms on the right-hand side of \eqref{eq: ewr3}, to derive
\begin{align}
2\langle \bz_{k+1} - \bar\bz_k, \bz^\star-\bz_{k+1/2} \rangle
&= \beta_k \| \bz^\star-\bz_0\|^2 + (1-\beta_k) \| \bz^\star-\bz_k\|^2 - \| \bz^\star-\bz_{k+1}\|^2\notag \\
&\quad- \frac{2\tau_k}{1+\tau_k} \left(\beta_k\|\bz_0-\bz_{k+1/2}\|^2 + (1-\beta_k)\|\bz_k-\bz_{k+1/2}\|^2 \right)\notag  \\
&\quad - \frac{\beta_k(1-\tau_k)}{1+\tau_k} \| \bz_0 - \bz_{k+1}\|^2 - \frac{(1-\beta_k)(1-\tau_k)}{1+\tau_k} \|\bz_k-\bz_{k+1}\|^2\notag \\
&\quad+ \frac{2\tau_k\gamma_k}{1+\tau_k} \langle \bg_k - \gtil(\bz_{k+1/2}, \xi_{k+1/2}), \bz_{k+1} - \bz_{k+1/2} \rangle,\label{eq: hkd3}
\end{align}
by using the identity $\|\beta \ba + (1-\beta) \bb \|^2 = \beta \| \ba\|^2 + (1-\beta) \| \bb\|^2 - \beta(1-\beta)\|\ba-\bb\|^2$ three times for all the terms involving $\bar\bz_k$ in \eqref{eq: ewr3}.

Same as the derivation of \citep[Eq. (4.27)]{alacaoglu2025towards}, one can estimate by Young's inequality that
\begin{align}
&\mathbb{E}\left[\frac{2\tau_k\gamma_k}{1+\tau_k} \langle \bg_k - \gtil(\bz_{k+1/2}, \xi_{k+1/2}), \bz_{k+1} - \bz_{k+1/2} \rangle\right] \notag \\
&\leq \frac{\tau_k}{(1+\tau_k)c_1} \mathbb{E}\| \bz_k - \bz_{k+1/2}\|^2 + \frac{\tau_k\gamma_k^2 c_1}{1+\tau_k} \E\| \bg_k - G(\bz_{k+1/2})\|^2 \notag \\
&\quad + \frac{1}{(1+\tau_k)c_2} \mathbb{E}\|\bz_{k+1} - \bz_k\|^2 + \frac{\tau_k^2\gamma_k^2c_2}{(1+\tau_k)} \mathbb{E}\| \bg_k - \gtil(\bz_{k+1/2},\xi_{k+1/2})\|^2,\label{eq: qeq3}
\end{align}
where we used
\begin{align*}
&\mathbb{E}_{k+1/2} \langle \bg_k - \gtil(\bz_{k+1/2}, \xi_{k+1/2}), \bz_{k+1} - \bz_{k+1/2} \rangle \\
&= \mathbb{E}_{k+1/2} [\langle \bg_k - \gtil(\bz_{k+1/2}, \xi_{k+1/2}), \bz_{k} - \bz_{k+1/2} \rangle +  \langle \bg_k - \gtil(\bz_{k+1/2}, \xi_{k+1/2}), \bz_{k+1} - \bz_{k} \rangle] \\
&= \langle \bg_k - G(\bz_{k+1/2}), \bz_{k} - \bz_{k+1/2} \rangle + \mathbb{E}_{k+1/2} \langle \bg_k - \gtil(\bz_{k+1/2}, \xi_{k+1/2}), \bz_{k+1} - \bz_{k} \rangle].
\end{align*}
On \eqref{eq: qeq3}, after Young's inequality and Lipschitzness of $G$, we estimate as
\begin{align}
&\mathbb{E}\left[\frac{2\tau_k\gamma_k}{1+\tau_k} \langle \bg_k - \gtil(\bz_{k+1/2}, \xi_{k+1/2}), \bz_{k+1} - \bz_{k+1/2} \rangle\right] \notag \\
&\leq \frac{\tau_k}{(1+\tau_k)c_1} \mathbb{E}\| \bz_k - \bz_{k+1/2}\|^2 + \frac{A_1L^2\tau_k\gamma_k^2 c_1}{1+\tau_k} \mathbb{E}\| \bz_k - \bz_{k+1/2}\|^2 + \frac{A_2\tau_k\gamma_k^2 c_1}{1+\tau_k} \mathbb{E}\| \bg_k - G(\bz_k)\|^2 \notag \\
&\quad + \frac{1}{(1+\tau_k)c_2} \mathbb{E}\|\bz_{k+1} - \bz_k\|^2 + \frac{3\tau_k^2\gamma_k^2c_2}{(1+\tau_k)} \mathbb{E}\| \bg_k - G(\bz_k)\|^2 \notag \\
&\quad+ \frac{3L^2\tau_k^2\gamma_k^2c_2}{(1+\tau_k)} \mathbb{E}\| \bz_k - \bz_{k+1/2}\|^2 + \frac{3\tau_k^2\gamma_k^2c_2}{(1+\tau_k)} \mathbb{E}\| G(\bz_{k+1/2}) - \gtil(\bz_{k+1/2},\xi_{k+1/2})\|^2.\label{eq: hkd4}
\end{align}
where $A_1 = (1+a_1),\, A_2 = (1+\frac{1}{a_1})$ for any positive $a_1$.

Applying \Cref{asp: 3} results in the estimate
\begin{align}
\mathbb{E}\| G(\bz_{k+1/2}) - \gtil(\bz_{k+1/2},\xi_{k+1/2})\|^2 &\leq B^2\mathbb{E}\| \bz_{k+1/2} - \bz_0 \|^2+ \sigma^2.
\label{eq: sds3}
\end{align}
We also have by Young's inequalities that
\begin{align}
&\frac{\rho \tau_k}{\gamma_k} \| \bar\bz_k - \bz_{k+1/2} - \gamma_k \bg_k + \gamma_k G(\bz_{k+1/2}) \|^2 \notag \\
&\leq  \left(\frac{E_1\rho \tau_k(1-\beta_k)}{\gamma_k} + E_3L^2\rho \tau_k\gamma_k \right) \| \bz_k-\bz_{k+1/2} \|^2 \notag \\
&\quad+ \frac{E_1\rho \tau_k\beta_k}{\gamma_k} \| \bz_0 - \bz_{k+1/2} \|^2 + E_2\rho \tau_k\gamma_k \| \bg_k - G(\bz_{k}) \|^2,\label{eq: jwz3}
\end{align}
where we use the property
\begin{align*}
&\|\bar\bz_k-\bz_{k+1/2}\|^2 \\
&= \|\beta_k \bz_0 + (1-\beta_k) \bz_k - \bz_{k+1/2}\|^2 \leq \beta_k\|\bz_0-\bz_{k+1/2}\|^2 
+ (1-\beta_k)\|\bz_k-\bz_{k+1/2}\|^2.
\end{align*}
On \eqref{eq: jwz3}, we have $E_1 = (1+e_1),\, E_2 = (1+\frac{1}{e_1})(1+e_2),$ and $E_3 = (1+\frac{1}{e_1})(1+\frac{1}{e_2})$ for any positive $e_1, e_2$.

Finally, we set
\begin{align*}
    c_1 = 4, ~~~ c_2 = 24, ~~~ a_1 = \frac{1}{8}, ~~~ e_1 = \frac{1}{4}, ~~~ e_2 = 10.
\end{align*}
With these choices, we have
\begin{align*}
    A_1 = \frac{9}{8}, ~~~ A_2 = 9, ~~~ E_1 = \frac{5}{4},~~~ E_2 = 55, ~~~ E_3 = \frac{11}{2}.
\end{align*}
Next, we multiply both sides of \eqref{eq: nbw3} by $2$, then plug in the expectation of \eqref{eq: hkd3}, \eqref{eq: hkd4}, \eqref{eq: sds3}, and \eqref{eq: jwz3} to get the result.
\end{proof}
We now continue with the recursion of the STORM estimator of \cite{cutkosky2019momentum} that is used to update $\bg_k$. This lemma is taken from \citep[Lemma 6.1]{alacaoglu2025towards} (which we refer to for the proof of this precise statement) that followed the idea of \cite{cutkosky2019momentum} with the minor change of using the \Cref{asp: 3} rather than the uniformly bounded variance.
\begin{lemma}\label{lem: vr_storm_lem}(See \cite{cutkosky2019momentum} for the original idea and \citep[Lemma 6.1]{alacaoglu2025towards} for the statement used here)
Let $\bg_k$ be as defined in \Cref{alg: var_red_fbf}. Then, under Assumptions \ref{asp: 1} and \ref{asp: 3}, we have
\begin{align*}
\frac{\alpha_k}{2} \E \| \bg_k - G(\bz_k)\|^2 &\leq \left( 1- \frac{\alpha_k}{2} \right) \E \| \bg_k - G(\bz_k)\|^2 - \E \| \bg_{k+1} - G(\bz_{k+1})\|^2\\
&\quad+2L^2\E \| \bz_{k+1} - \bz_k\|^2 + 2\alpha_k^2 (B^2\E\|\bz_k-\bz_0\|^2 + \sigma^2).
\end{align*}
\end{lemma}
We now include the restatement of \Cref{th: vr} that includes the parameter details and then provide its proof. 
\begin{theorem}\label{th: vr_supp}(Detailed restatement of \Cref{th: vr})\label{th: vr_supp}
Let Assumptions \ref{asp: 1}, \ref{asp: 2}, \ref{asp: 3} hold and suppose that $\rho > 0$ is sufficiently small with
\begin{align}\label{eq: sofx4}
    \rho \leq \min \Bigg\{ 
    \frac{L}{55} \left( \frac{83}{24L^2} - \frac{9{\bar\tau}}{\sqrt{3}L^2} \right),~~
    \frac{1}{12L}, ~~
    \frac{1}{16L(1+{\bar\tau})} - \frac{2\bar\tau}{17\sqrt{3}L}\left( \frac{9B^2}{L^2}+3\right)
    \Bigg\}=: f(\bar\tau).
\end{align}
Then, for the output of \Cref{alg: var_red_fbf} with 
\begin{equation*}
\beta_k = \frac{1}{k+3}, ~~~ \alpha_k = \frac{2}{\sqrt{k+3}}, ~~~ \gamma_k = \frac{1}{4L}, ~~~ \tau_k = \frac{\bar\tau}{\sqrt{k+3}},
\end{equation*}
where $\bar\tau \leq \min\left\{\frac{L^2}{219B^2},\frac{L^2}{20B^2+7L^2}\right\}$.
Then, we have
\begin{equation*}
\E [\res(\bz^{\out})] \leq \varepsilon \text{~with stochastic oracle complexity~} \widetilde{O}(\varepsilon^{-4}),
\end{equation*}
where $\Pr(\hat k=k) = \frac{\tau_k(k+3)}{\sum_{i=0}^{K-1} \tau_i(i+3)}$ and $\bz^{\out} = \bz_{\hat k+1/2}$.
\end{theorem}
\begin{remark}\label{eq: svx3}
As the result in \Cref{sec: mb_main}, the range of $\rho$ is quite restrictive. In particular, as $\bar\tau\to 0$ the dominant term in the upper bound, $\ \frac{1}{16L(1+\bar\tau)} - \frac{2\bar\tau}{17\sqrt{3}L}\left( \frac{9B^2}{L^2}+3\right)$, becomes $\approx\frac{1}{16L}$. 

To show that all the terms in the upper bound of $\rho$ are strictly positive, we require,
\begin{align*}
    \frac{1}{16(1+\bar\tau)} - \frac{2\bar\tau}{17\sqrt{3}}\left( \frac{9B^2}{L^2}+3\right) > 0
    \iff
    \frac{17\sqrt{3}L^2}{288B^2+96L^2} > \bar\tau + \bar\tau^2
\end{align*}
Because $\frac{L^2}{20B^2+7L^2} \geq \bar\tau$, this is satisfied.
It is easy to show that the other arguments in the definition of $\rho$, in \eqref{eq: sofx4}, are strictly positive.

To get the best dependence, one would also have to use the idea of \cite{pethick2023solving} to include another term in the potential function to have a bound for $\rho$ that approaches $\frac{1}{2L}$ as $\bar\tau\to 0$. We do not pursue this here for brevity since the best-known upper bound for $\rho$ (which is $1/L$) is attained in \Cref{sec: mlmc} for another algorithm.
\end{remark}
\begin{proof}[Proof of \Cref{th: vr_supp}]
To upper bound the last term on the right-hand side of the assertion of Lemma~\ref{lem: vr_one_it}, we use the result of Lemma~\ref{lem: vr_storm_lem}. In particular, in Lemma~\ref{lem: vr_storm_lem}, we let $\alpha_k = \frac{2}{\sqrt{k+3}}$ and multiply both sides by $\frac{\bar\tau(1-\beta_k)c_3}{L^2}$ to get
\begin{align*}
&\frac{\bar\tau(1-\beta_k)c_3}{L^2} \E \| \bg_{k+1} - G(\bz_{k+1}) \|^2 \\ &\leq \frac{\bar\tau(1-\beta_k)c_3}{L^2}\left( 1- \frac{1}{\sqrt{k+3}} \right) \E \| \bg_{k} - G(\bz_{k}) \|^2 - \frac{{\bar\tau}(1-\beta_k)c_3\alpha_k}{2L^2} \E\|\bg_k- G(\bz_k)\|^2 \notag \\
&\quad+2{\bar\tau}(1-\beta_k)c_3 \E \| \bz_{k+1} - \bz_k\|^2 \notag \\
&\quad+ \frac{2{\bar\tau}(1-\beta_k)c_3\alpha_k^2B^2}{L^2}\E\|\bz_k-\bz_0\|^2 +\frac{2{\bar\tau}(1-\beta_k)c_3\alpha_k^2\sigma^2}{L^2}.
\end{align*}
Using the identity $(1-\beta_k)\left(1-\frac{1}{\sqrt{k+3}} \right) \leq \frac{k+1}{k+3}$ and Young's inequality on the last inequality gives us that
\begin{align}
&\frac{{\bar\tau}(1-\beta_k)c_3}{L^2} \E \| \bg_{k+1} - G(\bz_{k+1}) \|^2 \notag \\ &\leq \frac{{\bar\tau}c_3(k+1)}{L^2(k+3)} \E \| \bg_{k} - G(\bz_{k}) \|^2 - \frac{{\bar\tau}(1-\beta_k)c_3\alpha_k}{2L^2} \E\|\bg_k- G(\bz_k)\|^2 \notag \\
&\quad+2{\bar\tau}(1-\beta_k)c_3 \E \| \bz_{k+1} - \bz_k\|^2 \notag \\
&\quad+ \frac{4{\bar\tau}(1-\beta_k)c_3\alpha_k^2B^2}{L^2}\left(\E\|\bz_{k+1}-\bz_0\|^2+\E\|\bz_k-\bz_{k+1}\|^2\right) + \frac{2{\bar\tau}(1-\beta_k)c_3\alpha_k^2\sigma^2}{L^2}.\label{eq: eon5}
\end{align}
Let us also note for the sixth term on the right-hand side of \eqref{eq: sfx4} that
\begin{align}
\mathcal{C}_{3,k} \E\| \bz_0-\bz_{k+1/2}\|^2 &= \left(\frac{5\rho\tau_k\beta_k}{2\gamma_k} + \frac{2\tau_k(36\tau_k\gamma_k^2 B^2 -\beta_k)}{1+\tau_k}\right)\E\| \bz_0-\bz_{k+1/2}\|^2 \notag \\
&\leq \left(\frac{5\rho\tau_k\beta_k}{2\gamma_k} - \frac{2\tau_k\beta_k}{1+\tau_k}\right)\E\| \bz_0-\bz_{k+1/2}\|^2\notag \\
&\quad+\frac{3\times 72\tau_k^2\gamma_k^2 B^2}{1+\tau_k}\mathbb{E}\left( \| \bz_0-\bz_{k+1}\|^2 + \| \bz_{k+1} - \bz_k\|^2 + \| \bz_k - \bz_{k+1/2}\|^2 \right).\label{eq: eon98}
\end{align}
Adding \eqref{eq: eon5} to the result of Lemma~\ref{lem: vr_one_it} and using \eqref{eq: eon98}, we obtain
\begin{align}
&\E \| \bz^\star - \bz_{k+1}\|^2 + \frac{{\bar\tau}(1-\beta_k)c_3}{L^2} \E \| \bg_{k+1} - G(\bz_{k+1}) \|^2 \notag \\
&\leq (1-\beta_k) \E\|\bz^\star- \bz_k\|^2+ \frac{{\bar\tau}c_3(k+1)}{L^2(k+3)} \E \| \bg_{k} - G(\bz_{k}) \|^2+\beta_k\|\bz^\star - \bz_0\|^2+ \mathcal{R}_k \notag \\
&\quad + \mathcal{D}_{1,k} \E\|\bz_k-\bz_{k+1/2}\|^2 + \mathcal{D}_{2,k} \E\| \bz_k-\bz_{k+1}\|^2  + \mathcal{D}_{3,k} \E\| \bz_0-\bz_{k+1/2}\|^2 \notag \\  &\quad
+ \mathcal{D}_{4,k} \E\|\bg_k - G(\bz_k)\|^2 + \mathcal{D}_{5,k} \E\|\bz_0-\bz_{k+1}\|^2,\label{eq: qqq4}
\end{align}
where the coefficients are defined as (where we assigned $c_3=6$ for the free variable)
\begin{align*}
    \mathcal{R}_k &= \mathcal{E}_k + \frac{12{\bar\tau}(1-\beta_k)\alpha_k^2\sigma^2}{L^2}, \\
    \mathcal{D}_{1,k} &= \mathcal{C}_{1, k}+\frac{216\tau_k^2\gamma_k^2B^2}{1+\tau_k} , \\
    \mathcal{D}_{2,k} &= \mathcal{C}_{2, k}  + 12{\bar\tau}(1-\beta_k) + \frac{24{\bar\tau}(1-\beta_k)\alpha_k^2B^2}{L^2} +\frac{216\tau_k^2\gamma_k^2B^2}{1+\tau_k}, \\
    \mathcal{D}_{3,k} &= \mathcal{C}_{3,k} - \frac{72\tau_k^2\gamma_k^2 B^2}{1+\tau_k}, \\
    \mathcal{D}_{4,k} &= \mathcal{C}_{4,k}  - \frac{3{\bar\tau}(1-\beta_k)\alpha_k}{L^2}, \\
    \mathcal{D}_{5,k} &= \frac{24{\bar\tau}(1-\beta_k)\alpha_k^2B^2}{L^2} + \frac{216\tau_k^2\gamma_k^2B^2}{1+\tau_k} - \frac{\beta_k(1-\tau_k)}{1+\tau_k}.
\end{align*}
This is, unfortunately, an even more complicated bound than Lemma~\ref{lem: vr_one_it}. One can see that the first two lines of \eqref{eq: qqq4} contain terms that will telescope after minor manipulations. The third line of \eqref{eq: qqq4} contains the term $\mathcal{R}_k$, which scales as $O(\tau_k^2+\alpha_k^2) = O(1/k)$. What remains is to select the parameters such that the last five terms of \eqref{eq: qqq4} will be nonpositive.

We wish to upper bound the sum of the last five terms in \eqref{eq: qqq4} with
\begin{align*}
\Theta(-\tau_k\E[(1-\beta_k) \| \bz_k-\bz_{k+1/2}\|^2 + (1-\beta_k) \| \bg_k - G(\bz_k)\|^2 + \beta_k \| \bz_0 - \bz_{k+1/2}\|^2]).
\end{align*}
Let us recall
\begin{align*}
\alpha_k = \frac{2}{\sqrt{k+3}},~~~ \gamma_k = \frac{1}{4L} ~~~ \tau_k = \frac{{\bar\tau}}{\sqrt{k+3}}.
\end{align*}
Then the last term of \eqref{eq: qqq4} will be nonpositive, that is, $\mathcal{D}_{5, k}\leq 0$, if,
\begin{align*}
\frac{96{\bar\tau}(1-\beta_k)B^2}{(k+3)L^2} + \frac{27{\bar\tau}^2B^2}{2(k+3)(1+\tau_k)L^2} \leq \frac{\beta_k(1-\tau_k)}{1+\tau_k},
\end{align*}
where $\beta_k = \frac{1}{k+3}$, ${\bar\tau}^2\leq {\bar\tau}$ because ${\bar\tau}\leq 1$ and also $1-\tau_k\geq 4/5$ since ${\bar\tau}\leq 1/3$. This will be implied by the definition of $\tau_k$ since
\begin{align*}
{\bar\tau} \leq \frac{L^2}{219B^2}.
\end{align*}
Let us now also note
\begin{equation*}
1-\beta_k \leq 1 \iff \frac{1}{1-\beta_k} \geq 1 \text{~and~} \frac{1-\tau_k}{1+\tau_k} \geq \frac{1}{2} \text{~when~} {\bar\tau} \leq \frac{1}{3} \text{~and~} 1-\tau_k \geq 1/2 \text{~when~} {\bar\tau}\leq 1.
 \end{equation*}
 We estimate the second from last term of \eqref{eq: qqq4}. By using $\mathcal{C}_{4,k}$ from \eqref{eq: c_defs}, we have
 \begin{align*}
 \mathcal{D}_{4, k} &=  110\rho\tau_k\gamma_k +\frac{4\tau_k\gamma_k^2(9+18\tau_k)}{1+\tau_k}  - \frac{3{\bar\tau}(1-\beta_k)\alpha_k}{L^2} \\
 &= 110\rho\tau_k\gamma_k +\frac{4\tau_k\gamma_k^2(9+18\tau_k)}{1+\tau_k}  - \frac{6\tau_k(1-\beta_k)}{L^2}.
 \end{align*} 
We will upper bound this term by $-\frac{\tau_k(1-\beta_k)}{32L^2}$. That is, we have
 \begin{align*}
 \mathcal{D}_{4, k} \leq -\frac{\tau_k(1-\beta_k)}{32L^2} \iff 
 \frac{55\rho}{2L} + \frac{9}{4L^2(1+\tau_k)} + \frac{9{\bar\tau}}{2L^2(1+\tau_k)\sqrt{k+3}} \leq \frac{191(1-\beta_k)}{32L^2},
 \end{align*}
 where for the equivalence, we divided both sides of the inequality by $\tau_k$ and plugged in $\gamma_k$.
 
Since $\frac{1}{1+\tau_k} \leq 1$, $1-\beta_k\geq2/3$, and $\frac{1}{\sqrt{k+3}}\leq 1/\sqrt{3}$, this will be implied by
\begin{align}
\rho \leq \frac{L}{55} \left( \frac{83}{24L^2} - \frac{9{\bar\tau}}{\sqrt{3}L^2} \right).\label{eq: rho_bd1}
\end{align}
We estimate the third from last term of \eqref{eq: qqq4}. We wish to show that $\mathcal{D}_{3,k} \leq -\frac{\tau_k\beta_k}{3(1+\tau_k)}$. By using $\mathcal{C}_{3,k}$ from \eqref{eq: c_defs}, we have
\begin{equation*}
\mathcal{D}_{3, k} = \frac{5\rho\tau_k\beta_k}{2\gamma_k} - \frac{2\tau_k\beta_k}{1+\tau_k} \leq -\frac{\tau_k\beta_k}{3(1+\tau_k)},
\end{equation*}
when
\begin{align}
\rho \leq \frac{1}{12L}.\label{eq: rho_bd2}
\end{align}
We used here that $1+\tau_k \geq 2$ since $\tau_k\leq {\bar\tau}\leq 1$.

We estimate the sixth term in \eqref{eq: qqq4}. We wish to show that $\mathcal{D}_{2, k} \leq 0$. Let us use the definition of $\mathcal{C}_{2, k}$ from \eqref{eq: c_defs} to write
\begin{align*}
\mathcal{D}_{2, k} &= \frac{1}{24(1+\tau_k)} + 12{\bar\tau}(1-\beta_k) + \frac{24{\bar\tau}(1-\beta_k)\alpha_k^2B^2}{L^2} +\frac{216\tau_k^2\gamma_k^2B^2}{1+\tau_k} - \frac{(1-\tau_k)(1-\beta_k)}{1+\tau_k}.
\end{align*}
Then, we have
\begin{align*}
\mathcal{D}_{2,k} \leq 0 \Longleftarrow 12{\bar\tau}(1-\beta_k) + \frac{96{\bar\tau}(1-\beta_k)B^2}{(k+3)L^2} + \frac{27{\bar\tau}^2B^2}{2(k+3)(1+\tau_k)L^2} \leq \frac{59}{120(1+\tau_k)},
\end{align*}
because $1-\beta_k \geq 2/3$, $\frac{1}{1+\tau_k} \geq 1/2$, and $1-\tau_k \geq 4/5$. This is implied by
\begin{align*}
{\bar\tau} \leq \frac{59}{120}\left( \frac{L^2}{73B^2+24L^2} \right).
\end{align*}
This is because $1-\beta_k\leq 1$, $\frac{1}{k+3}\leq\frac{1}{3}$,  ${\bar\tau}\leq 1/3$ (so, ${\bar\tau}^2\leq{\bar\tau}$).

We next estimate the fifth term in \eqref{eq: qqq4}. Let us use the definition of $\mathcal{C}_{1, k}$ from \eqref{eq: c_defs} to write
\begin{align*}
\mathcal{D}_{1, k} &=  \frac{\tau_k(1/4+9/2\gamma_k^2L^2 + 72\tau_k\gamma_k^2 L^2 - 2(1-\beta_k))}{1+\tau_k}+\frac{5\rho\tau_k(1-\beta_k)}{2\gamma_k}+{11\rho\tau_kL^2\gamma_k}+\frac{216\tau_k^2\gamma_k^2B^2}{1+\tau_k}.
\end{align*}
We then have that
\begin{align*}
\mathcal{D}_{1,k} \leq -\frac{\tau_k(1-\beta_k)}{128(1+\tau_k)} &\iff
\frac{1}{4(1+\tau_k)} + 10L\rho(1-\beta_k) + \frac{11\rho L}{4} + \frac{9}{32(1+\tau_k)} \\
&\quad+ \frac{{\bar\tau}}{\sqrt{k+3}(1+\tau_k)}\left( \frac{27B^2}{2L^2}+\frac{9}{2} \right) \leq \frac{255(1-\beta_k)}{128(1+\tau_k)},
\end{align*}
where we divided both sides by $\tau_k$ and used $\gamma_k=\frac{1}{4L}$ in the equivalence step. 

Because $1-\beta_k\geq 2/3$, the previous inequality will be implied by
\begin{align*}
10L\rho(1-\beta_k) + \frac{11\rho L}{4} + \frac{{\bar\tau}}{\sqrt{k+3}(1+\tau_k)}\left( \frac{27B^2}{2L^2}+\frac{9}{2} \right) \leq \frac{51}{64(1+\tau_k)}.
\end{align*}
Then, by using $1-\beta_k\leq 1$, $\sqrt{k+3} \geq \sqrt{3}$, and $\frac{1}{1+\tau_k} \leq 1$, this is true as long as
\begin{align*}
\rho \leq  \frac{1}{16L(1+\tau_k)} - \frac{2{\tau_k}}{17\sqrt{3}L}\left( \frac{9B^2}{L^2}+3 \right),
\end{align*}
which can be easily made independent of $k$ since $\tau_k$ is nonincreasing, that is, this bound of $\rho$ is implied by
\begin{align}
\rho \leq \frac{1}{16L(1+{\bar\tau})} - \frac{2\bar\tau}{17\sqrt{3}L}\left( \frac{9B^2}{L^2}+3\right),\label{eq: rho_bd3}
\end{align}
where the upper bound is guaranteed to be positive due to the definition of $\bar\tau$.

Summarizing the constraints derived from the nonpositivity of $\mathcal{D}_{4,k}$, $\mathcal{D}_{3,k}$, and $\mathcal{D}_{1,k}$, we require the parameter $\rho$ to satisfy the following condition,
\begin{align*}
    \rho \leq \min \Bigg\{ 
    \frac{L}{55} \left( \frac{83}{24L^2} - \frac{9{\bar\tau}}{\sqrt{3}L^2} \right),~~
    \frac{1}{12L}, ~~
    \frac{1}{16L(1+{\bar\tau})} - \frac{2\bar\tau}{17\sqrt{3}L}\left( \frac{9B^2}{L^2}+3\right)
    \Bigg\}.
\end{align*}
where $\bar\tau \leq \min\left\{\frac{L^2}{219B^2}, \frac{L^2}{20B^2+7L^2}\right\}$. 

As $\bar\tau \to 0$, the bound in \eqref{eq: rho_bd1} approaches $\frac{1}{15.9L}$, while the bound in \eqref{eq: rho_bd3} approaches $\frac{1}{16L}$, so this is the bound we have in our
theorem statement for $\bar\tau\to 0$.

With these, we then estimate \eqref{eq: qqq4} as
\begin{align*}
&\E \| \bz^\star - \bz_{k+1}\|^2 + \frac{6{\bar\tau}(1-\beta_k)}{L^2} \E \| \bg_{k+1} - G(\bz_{k+1}) \|^2 \\
&\leq (1-\beta_k) \E\|\bz^\star- \bz_k\|^2+ \frac{6{\bar\tau}(k+1)}{L^2(k+3)} \E \| \bg_{k} - G(\bz_{k}) \|^2+\beta_k\|\bz^\star - \bz_0\|^2 \\
&\quad + \frac{72\tau_k^2\gamma_k^2\sigma^2}{(1+\tau_k)} + \frac{12{\bar\tau}(1-\beta_k)\alpha_k^2\sigma^2}{L^2}\\
&\quad -\frac{\tau_k(1-\beta_k)}{128(1+\tau_k)}\E \|\bz_k-\bz_{k+1/2}\|^2 -\frac{\tau_k\beta_k}{3(1+\tau_k)} \E \| \bz_0-\bz_{k+1/2}\|^2 - \frac{\tau_k(1-\beta_k)}{32L^2} \E \| \bg_k - G(\bz_k)\|^2.
\end{align*}
We now multiply both sides by $k+3$ to obtain
\begin{align*}
&(k+3)\E \| \bz^\star - \bz_{k+1}\|^2 + \frac{6{\bar\tau}(k+2)}{L^2} \E \| \bg_{k+1} - G(\bz_{k+1}) \|^2 \\
&\leq (k+2) \E\|\bz^\star- \bz_k\|^2+ \frac{6{\bar\tau}(k+1)}{L^2} \E \| \bg_{k} - G(\bz_{k}) \|^2+\|\bz^\star - \bz_0\|^2 \\
&\quad + (k+3)\left(\frac{72\tau_k^2\gamma_k^2\sigma^2}{(1+\tau_k)} + \frac{12{\bar\tau}(1-\beta_k)\alpha_k^2\sigma^2}{L^2}\right)\\
&\quad -\frac{\tau_k(k+3)(1-\beta_k)}{128(1+\tau_k)}\E \|\bz_k-\bz_{k+1/2}\|^2-\frac{\tau_k(k+3)\beta_k}{3(1+\tau_k)} \E \| \bz_0-\bz_{k+1/2}\|^2\\
&\quad - \frac{\tau_k(k+3)(1-\beta_k)}{32L^2} \E \| \bg_k - G(\bz_k)\|^2.
\end{align*}
That is, we have for some $\delta\geq\frac{1}{128}$ that
\begin{align*}
&\delta\tau_k(k+3)\E\left[\frac{1-\beta_k}{L^2} \|\bg_k - G(\bz_k)\|^2 + (1-\beta_k)\|\bz_k-\bz_{k+1/2}\|^2 + \beta_k \|\bz_0-\bz_{k+1/2}\|^2\right] \\
&\leq -(k+3)\E \| \bz^\star - \bz_{k+1}\|^2 - \frac{6{\bar\tau}(k+2)}{L^2} \E \| \bg_{k+1} - G(\bz_{k+1}) \|^2 \\
&\quad +(k+2) \E\|\bz^\star- \bz_k\|^2+ \frac{6{\bar\tau}(k+1)}{L^2} \E \| \bg_{k} - G(\bz_{k}) \|^2+ O(1).
\end{align*}
Summing up, using $\sum_{k=0}^K \tau_k(k+3) = \Omega(K^{3/2})$, since $\tau_k=\Theta(1/\sqrt{k})$, we obtain that 
\begin{align*}
&\frac{1}{\sum_{k=0}^K \tau_k(k+3)}\sum_{k=0}^{K-1}\tau_k(k+3)\E\bigg[\frac{1-\beta_k}{L^2} \|\bg_k - G(\bz_k)\|^2 + (1-\beta_k)\|\bz_k-\bz_{k+1/2}\|^2 \\
&\qquad\qquad\qquad\qquad\qquad\qquad\qquad\qquad\qquad\qquad+ \beta_k \|\bz_0-\bz_{k+1/2}\|^2\bigg]= O(K^{-1/2}),
\end{align*}
that is, the left-hand side is smaller than $\varepsilon^2$ after $O(\varepsilon^{-4})$ iterations. The left-hand side of the last inequality can be converted to $\res(\bz_{k+1/2})$ and after using \citep[Lemma 6.2]{alacaoglu2025towards}. With the same idea as \Cref{sec: mb_main} and \citep[Theorem 4.6]{alacaoglu2025towards}, we convert this to a guarantee on a randomly selected iterate.
\end{proof}

\newpage
\section{Details for Section \ref{sec: numer} and Further Numerical Results}\label{app: exper}

\subsection{Additional experiment}

\begin{figure}[h] 
    \centering 
    \subfigure{%
    \includegraphics[width=0.32\linewidth]{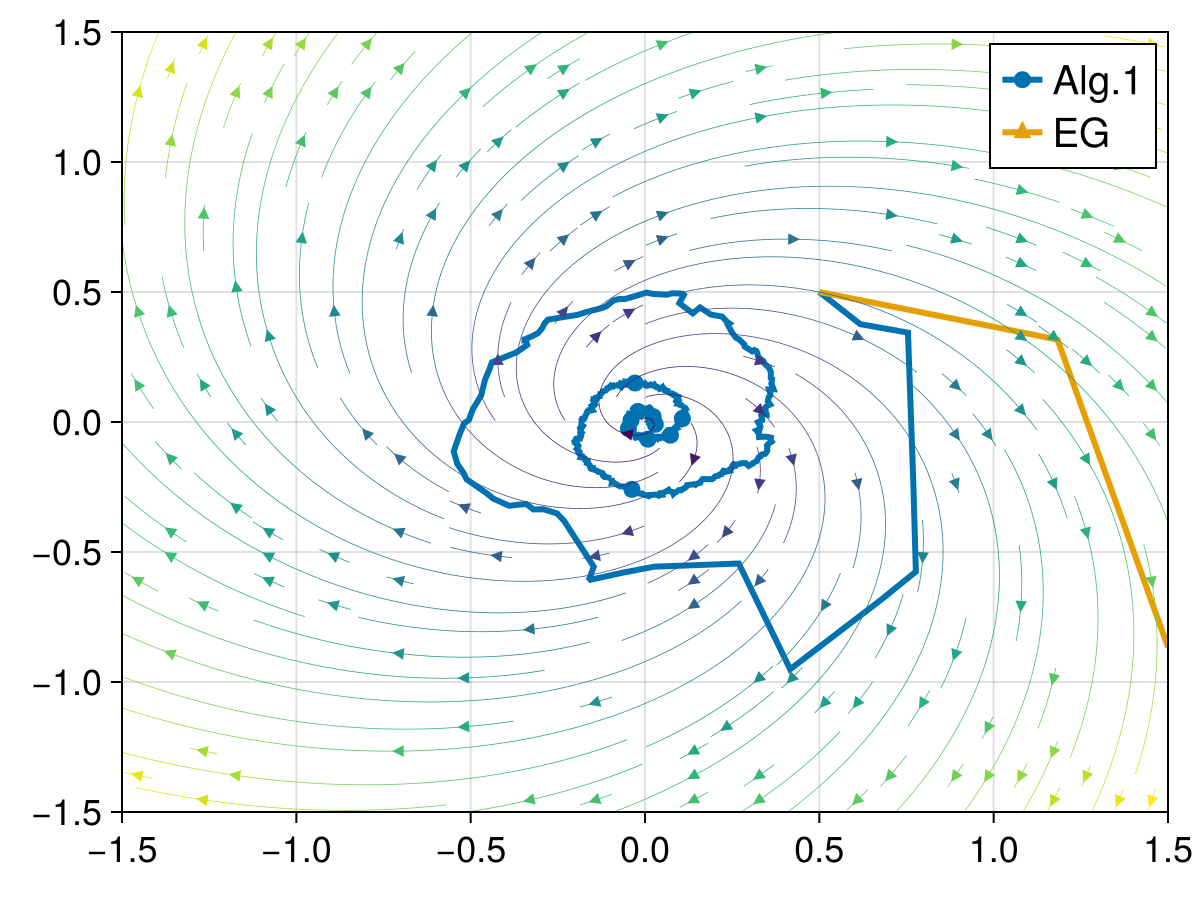}}\hfill 
    \subfigure{%
    \includegraphics[width=0.32\linewidth]{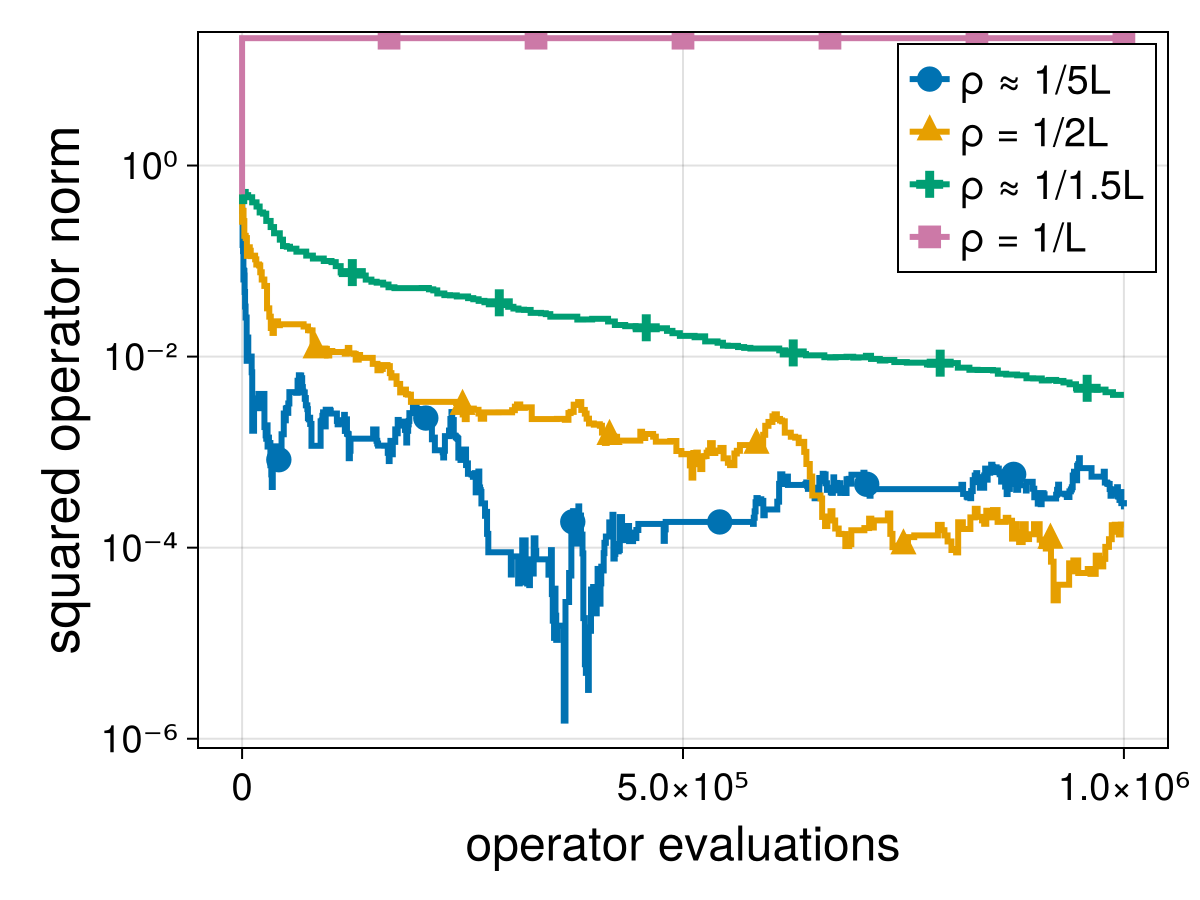}}\hfill 
    \subfigure{%
    \includegraphics[width=0.32\linewidth]{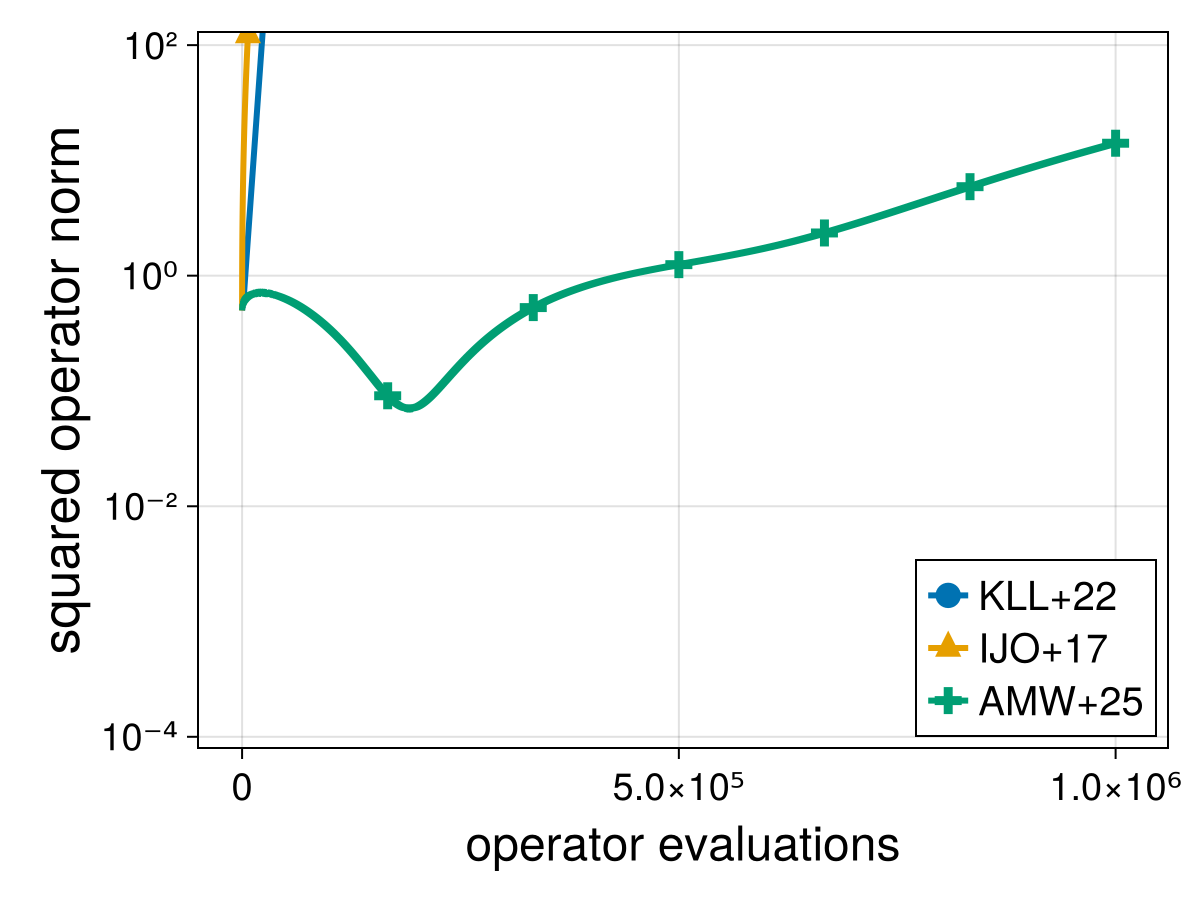}}
    \caption{\small 
    Left: Trajectories of Alg.1 and EG for the LAx counter-example (see \eqref{eq: def_rot}). 
    Middle: Alg.1 for counter-example with varying $\rho$. 
    Right: Methods from \cite{kotsalis2022simpleii}, \cite{iusem2017extragradient}, \cite{alacaoglu2025towards} in \Cref{tab:constraints} for the LAx counter-example.
    \emph{\small Middle and Right panel has log-scaled $y$-axis.}} 
\label{fig:lax-traj-regime} 
\end{figure}

We use the LAx counter-example instance from \citep[Theorem 4.3]{gorbunov2023convergence}. For $L > 0$ define
\begin{align*}
    F(x) = LAx, \qquad x\in\mathbb{R}^2,
\end{align*}
where $A$ is the rotation matrix
\begin{align}\label{eq: def_rot}
A =
\begin{pmatrix}
    \cos\theta & -\sin\theta\\
    \sin\theta & \cos\theta
\end{pmatrix},
\qquad \rho = -\tfrac{\cos\theta}{L}.
\end{align}

(Fig. \ref{fig:lax-traj-regime}, \emph{Left}) We use the same setup as the left panel of Fig.~\ref{fig:lax-gamma-png} fixing $\rho = \tfrac{1}{2L}$ to match the residual plot and keep the initialization and stopping criterion identical. In this setting, EG spirals outward and diverges as predicted by theory \citep{gorbunov2023convergence}, whereas \Cref{alg:weakmvi_stoc} converges to the solution and stabilizes. 

(Fig. \ref{fig:lax-traj-regime}, \emph{Middle}) On the same LAx instance, we examine stability as a function of the rotation angle $\theta$. We report the operator norm $\|F(z)\|^2$ on a log-scaled $y$-axis. Over the tested grid of $\theta$, \Cref{alg:weakmvi_stoc} is stable for $\rho < \tfrac{1}{L}$ and becomes unstable at $\rho = \tfrac{1}{L}$, confirming a stability boundary near $\tfrac{1}{L}$ for this instance, confirming our theoretical results.

(Fig. \ref{fig:lax-traj-regime}, \emph{Right}) For a fixed instance with $\rho=\tfrac{1}{2L}$, we also run the three method in \Cref{tab:constraints} that require $\rho = 0$ (\cite{kotsalis2022simpleii}, \cite{iusem2017extragradient}, \cite{alacaoglu2025towards}) using the same initialization where the noise in the stochastic gradient have the Gaussian distribution. In this setting, all three baselines become unstable and diverge since their theory only covers $\rho=0$.

\begin{figure}[h]
  \centering
  \includegraphics[width=0.42\linewidth]{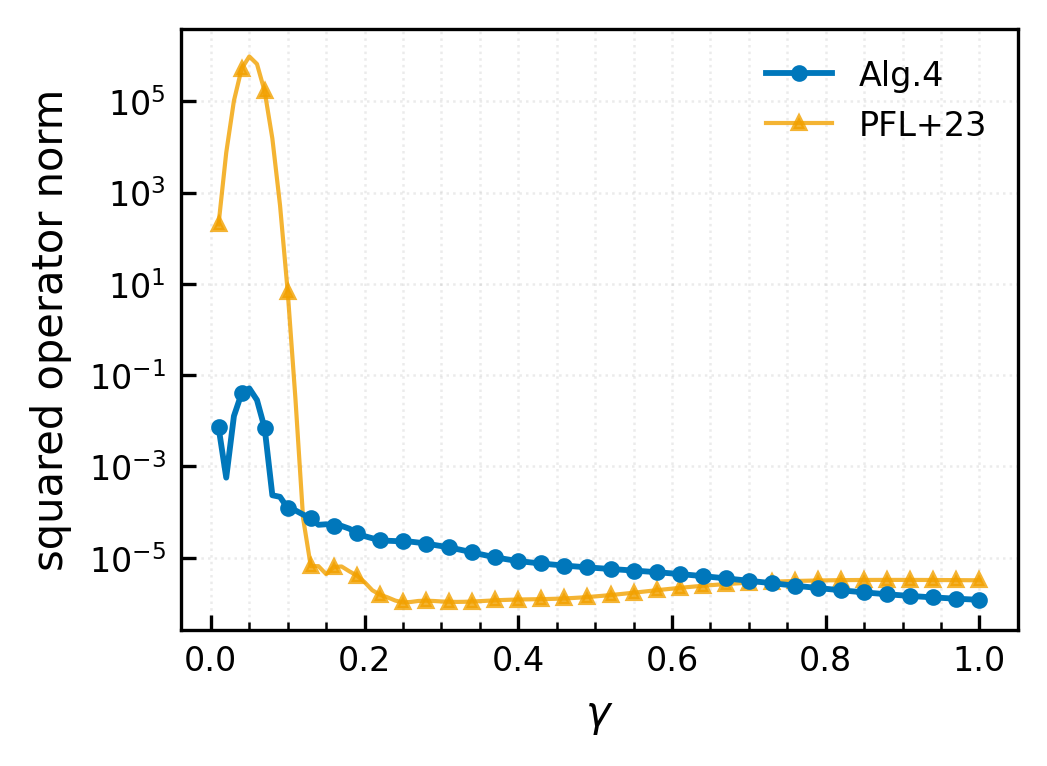}
  \caption{\small Alg. 4 and the algorithm of \cite{pethick2023solving} for the unconstrained problem in \eqref{eq: uq_prob} with Gaussian noise. \emph{Panel has log-scaled $y$-axis.}} 
  \label{fig:gauss-example2}
\end{figure}

In Fig. \ref{fig:gauss-example2}, and Fig.~\ref{fig:lax-gamma-png} (middle/right), we use the unconstrained quadratic problem from \citep[Example 2]{pethick2023solving},
\begin{equation}\label{eq: uq_prob}
    \min_{x\in\mathbb{R}} \max_{y\in\mathbb{R}} \phi(x, y) := axy + \frac{b}{2}x^2-\frac{b}{2} y^2
\end{equation}
where $a = \sqrt{L^2-L^4\rho^2}$ and $b = L^2\rho$. In our experiments, we set $L =1$ and $\rho = \frac{1}{10L}$. For Fig. \ref{fig:gauss-example2}, we replace the heavy-tailed noise (from Fig.~\ref{fig:lax-gamma-png} (middle/right)) with zero-mean Gaussian noise and keep the same $\gamma$ grid and seven-seed, reporting the mean of the last iterate. We reuse the \Cref{alg: var_red_fbf} schedule fixed in the main text. Both methods behave similarly for moderate $\gamma$; as it decreases, the method of \cite{pethick2023solving} loses stability and diverges, whereas \Cref{alg: var_red_fbf} remains stable.

\subsection{Experiment Details and Hyperparameters}

For the LAx counter-example with \Cref{alg:weakmvi_stoc} we tune the $\alpha_k$ schedule by sweeping the multiplier $c$ in $\alpha_k = \frac{c\,\alpha}{\sqrt{k+2}\,\log(k+3)}$, increasing $c$ from $1.0$ and enforcing $\alpha_0 < 1$. For the \cref{alg:fbf_mlmc} inner loop, we run a coarse-to-fine search over a single scaling coefficient that determines the per-iteration budgets $(N_k, M_k)$, targeting the smallest sample sizes that preserve the observed accuracy. Fig.~\ref{fig:lax-gamma-png} (left) and Fig.~\ref{fig:lax-traj-regime} (left) correspond to the same run. 

For Fig.~\ref{fig:lax-traj-regime} (middle) we reuse the $N_k, M_k$ value from the example above, for  \cref{alg:weakmvi_stoc}. We retune only the $\alpha_k$ multiplier $c$ using the same sweep and selection rule. At the boundary $\rho=\tfrac{1}{L}$, $\eta=\frac{1}{L}$, we need $\alpha=1-\tfrac{\rho}{\eta}=0$, so we set $\alpha=10^{-3}$ to initialize the $(N_k, M_k)$ budgets (since they depend on $\alpha$), but \cref{alg:weakmvi_stoc} still diverges at this threshold. Fig.~\ref{fig:lax-traj-regime} (right) For each method, we use the hyperparameter setting prescribed by its theory and show the method becomes unstable since $\rho\neq 0$.

Similar to \cref{alg:weakmvi_stoc}, we tune the $\alpha_k$ schedule for \cref{alg: var_red_fbf} with a coarse-to-fine grid over the initialization $\alpha_0$ and the decrease factor $c$, that is $\alpha_k = \frac{\alpha_0}{\sqrt{k/c+1}}$ in \cref{alg: var_red_fbf}. We first run a broad scan over $\alpha_0$ to identify a stable region, then perform a focused search around the best area to select the final setting.  The chosen $\alpha_0$ and $c$ are frozen for the full $\gamma$ sweep and reused across noise models. With this schedule, \cref{alg: var_red_fbf} remains stable as $\gamma$ decreases, while \cite{pethick2023solving} diverges.

\subsection{Computing infrastructure}

All experiments are ran locally on a MacBook Pro (Apple M2 Pro, 10-core CPU; macOS, arm64). We used Julia 1.10.5 and Python 3.8.20. No GPU acceleration was used; all results are CPU-only. We fixed pseudorandom seeds for each run and logged hyperparameters and metrics for reproducibility. Environment files  are included in the supplement.

\subsection{Code bases and modifications}

\emph{Julia (LAx / Alg. 1)}
We build on an open-source Julia package from 
\cite{alacaoglu2023beyond} available at \url{https://github.com/AxelBohm/beyond_golden_ratio.git}. While keeping its original structure, we implemented LAx counter-example problem and our \cref{alg:weakmvi_stoc}

\emph{Python (Example~2 / Alg. 4)}
We adapt our code-base from \cite{pethick2023solving} and its repository at \url{https://github.com/LIONS-EPFL/stochastic-weak-minty-code.git}. We use their method, labeled as PFL+23 with no change, and implemented \cref{alg: var_red_fbf} using their primitives in the codebase.

\end{document}